\documentclass[12pt]{article}
\usepackage{url}
\RequirePackage[colorlinks,citecolor=blue,urlcolor=blue,linkcolor=blue]{hyperref}

\usepackage[ansinew]{inputenc}

\usepackage{doi}

\usepackage{mathtools}

\usepackage{amsfonts}
\usepackage{amsthm}
\usepackage{amssymb}

\usepackage{color}
\usepackage{bm} 
\usepackage{dsfont} 

\usepackage{graphicx}
\usepackage[lmargin=3cm,rmargin=3cm,bottom=3cm,top=3cm]{geometry}

\usepackage{tikz}
\definecolor{mycolor1}{rgb}{0.105882,0.619608,0.466667}
\definecolor{mycolor2}{rgb}{0.85098,0.372549,0.00784314}
\definecolor{mycolor3}{rgb}{0.458824,0.439216,0.701961}
\definecolor{mycolor4}{rgb}{0.905882,0.160784,0.541176}
\definecolor{mycolor5}{rgb}{0.4,0.65098,0.117647}
\definecolor{mycolor6}{rgb}{0.65098,0.462745,0.113725}
\definecolor{mycolor7}{rgb}{0.901961,0.670588,0.00784314}
\definecolor{mycolor8}{rgb}{0.4,0.4,0.4}
\definecolor{mycolor9}{rgb}{0.301961,0,0.294118}
\definecolor{mycolor10}{rgb}{0.0313725,0.25098,0.505882}

\usetikzlibrary{calc}        

\makeatletter
\newif\ifmygrid@coordinates
\tikzset{/mygrid/step line/.style={line width=0.80pt,draw=gray!80},
         /mygrid/steplet line/.style={line width=0.25pt,draw=gray!80}}
\pgfkeys{/mygrid/.cd,
         step/.store in=\mygrid@step,
         steplet/.store in=\mygrid@steplet,
         coordinates/.is if=mygrid@coordinates}
\def\mygrid@def@coordinates(#1,#2)(#3,#4){%
    \def\mygrid@xlo{#1}%
    \def\mygrid@xhi{#3}%
    \def\mygrid@ylo{#2}%
    \def\mygrid@yhi{#4}%
}
\newcommand\DrawGrid[3][]{%
    \pgfkeys{/mygrid/.cd,coordinates=true,step=1,steplet=0.2,#1}%
    \draw[/mygrid/steplet line] #2 grid[step=\mygrid@steplet] #3;
    \draw[/mygrid/step line] #2 grid[step=\mygrid@step] #3;
    \mygrid@def@coordinates#2#3%
    \ifmygrid@coordinates%
        \draw[/mygrid/step line]
        \foreach \xpos in {\mygrid@xlo,...,\mygrid@xhi} {%
          (\xpos,\mygrid@ylo) -- ++(0,-3pt)
                              node[anchor=north] {$\xpos$}
        }
        \foreach \ypos in {\mygrid@ylo,...,\mygrid@yhi} {%
          (\mygrid@xlo,\ypos) -- ++(-3pt,0)
                              node[anchor=east] {$\ypos$}
        };
    \fi%
}
\makeatother

\usepackage{psfrag}

\usepackage{algorithm}
\usepackage{algpseudocode}
\algdef{SE}[DOWHILE]{Do}{doWhile}{\algorithmicdo}[1]{\algorithmicwhile\ #1}%

\usepackage{sectsty}
\makeatletter\def\@seccntformat#1{\protect\makebox[0pt][r]{\csname the#1\endcsname\hspace{12pt}}}\makeatother

\usepackage{mathrsfs} 

\usepackage[square,numbers,compress]{natbib}

\newcommand{\remove}[1]{}
\newcommand{\removesafe}[1]{}

\usepackage{subfigure}

\newcommand{\transpose}{^\top\! }

\newcommand{\inner}[2]{\left\langle{#1},{#2}\right\rangle}
\newcommand{\innersmall}[2]{\langle{#1},{#2}\rangle}

\newcommand{\T}{\mathrm{T}}

\newcommand{\Rd}{{\mathbb{R}^{d}}}

\newcommand{\reals}{{\mathbb{R}}}

\newcommand{\Rn}{{\mathbb{R}^n}}

\newcommand{\Cn}{\mathbb{C}^n}
\newcommand{\Cnn}{\mathbb{C}^{n\times n}}

\newcommand{\diag}{\mathrm{diag}}

\newcommand{\calC}{\mathcal{C}}

\newcommand{\calM}{\mathcal{M}}
\newcommand{\rank}{\operatorname{rank}}

\newcommand{\opnorm}[1]{\left\|{#1}\right\|_\mathrm{op}}

\newcommand{\frobnormsmall}[1]{\|{#1}\|_\mathrm{F}}

\newcommand{\lambdamax}{\lambda_\mathrm{max}}
\newcommand{\lambdamin}{\lambda_\mathrm{min}}

\newcommand{\ddiag}{\mathrm{ddiag}}

\newtheorem{theorem}{Theorem}
\newtheorem{lemma}[theorem]{Lemma}

\newtheorem{remark}{Remark}

\title{Nonconvex phase synchronization}

\author{Nicolas Boumal\thanks{Department of Mathematics, Princeton University, \texttt{nboumal@math.princeton.edu}.}}

\begin{document}

\maketitle

\begin{abstract}
We estimate $n$ phases (angles) from noisy pairwise relative phase measurements. The task is modeled as a nonconvex least-squares optimization problem. It was recently shown that this problem can be solved in polynomial time via convex relaxation, under some conditions on the noise. In this paper, under similar but more restrictive conditions, we show that a modified version of the power method converges to the global optimum. This is simpler and (empirically) faster than convex approaches. Empirically, they both succeed in the same regime. Further analysis shows that, in the same noise regime as previously studied, second-order necessary optimality conditions for this quadratically constrained quadratic program are also sufficient, despite nonconvexity.
\end{abstract}

\section{Introduction}

We consider the problem of estimating $n$ angles $\theta_1, \ldots, \theta_n$ in $[0, 2\pi)$ based on noisy measurements of their differences $\theta_i - \theta_j \textrm{ mod } 2\pi$. Equivalently, we aim to recover $n$ phases $e^{i\theta_1}, \ldots, e^{i\theta_n}$ from measurements of relative phases $e^{i(\theta_i - \theta_j)}$. 
This situation comes up notably in clock synchronization of distributed networks and signal reconstruction from phaseless measurements---see the \emph{related work} section below for references and more examples.

The target parameter is
\begin{align}
	z & \in \Cn_1 \triangleq \{ x \in \Cn : |x_1| = \cdots = |x_n| = 1 \}.
	\label{eq:Cn1}
\end{align}
Writing $\bar z_j = e^{-i\theta_j}$ for the complex conjugate of $z_j = e^{i\theta_j}$, the measurements are of the form $C_{ij} \approx z_i \bar z_j = e^{i(\theta_i - \theta_j)}$. They are stored in the Hermitian matrix
\begin{align}
	C & = zz^* + \Delta
	\label{eq:C}
\end{align}
where $z^*$ is the Hermitian conjugate of $z$ and $\Delta$ is a Hermitian perturbation. Motivated by the scenario where $\Delta$ contains white Gaussian noise, we focus on the associated maximum likelihood estimation problem, which corresponds to the least-squares estimator in the nonlinear space $\Cn_1$. Writing $\frobnormsmall{\cdot}$ for the Frobenius norm, this reads:
\begin{align*}
	\min_{x\in\Cn_1} \frobnormsmall{C - xx^*}^2.
\end{align*}
The Frobenius norm expands as $\frobnormsmall{C - xx^*}^2 = \frobnormsmall{C}^2 + \frobnormsmall{xx^*}^2 - 2 x^* C x$. Under the constraints, the first two terms are constant so that the problem is equivalent to our object of study:
\begin{align}
	\max_{x\in\Cn_1} f(x) = x^*Cx.
	\tag{P}
	\label{eq:P}
\end{align}
This is a smooth optimization problem on a manifold (a product of $n$ circles in the complex plane)~\citep{AMS08}. It is nonconvex and NP-hard~\cite[Prop.\,3.5]{zhang2006complex}.

In this paper, we study necessary and sufficient optimality conditions for~\eqref{eq:P} and propose a simple method which converges to a global optimum, under conditions on $\Delta$~\eqref{eq:C}. These conditions are met with high probability in the white Gaussian noise scenario, provided the variance is not too large.

Since measurements convey only relative information, the \emph{global phase} of $z$ is unidentifiable, that is, it cannot be known whether the sought signal is $z$ or $ze^{i\theta}$. This is reflected in the invariance $f(x) = f(xe^{i\theta})$ for all $\theta$. Accordingly, we define an equivalence relation $\sim$ over $\Cn_1$:
\begin{align}
	x \sim y \iff x = ye^{i\theta} \textrm{ for some } \theta \iff |x^*y| = n.
	\label{eq:equivalence}
\end{align}
This equivalence relation partitions $\Cn_1$ in subsets of indistinguishable signals called \emph{equivalence classes}. The equivalence class of $x$ is $[x] = \{ xe^{i\theta} \textrm{ for all } \theta \}$. The set of equivalence classes is the \emph{quotient space} $\Cn_1 /\!\! \sim$.
An adequate error measure (or distance on the quotient space) is
\begin{align}
	d(z, x) & = \min_{\theta \in \reals}\|xe^{i \theta}-z\|_2 = \sqrt{2(n-|z^*x|)}.
	\label{eq:distance}
\end{align}
When we say that the solution of~\eqref{eq:P} is unique \emph{up to phase}, we mean that the set of global optima is one equivalence class.

The following result from~\cite{bandeira2014tightness} states that global optima of~\eqref{eq:P} are close to the unknown signal $z$ (as compared to the maximal distance $\sqrt{2n}$) if
the perturbation $\Delta$ is small compared to the signal, in operator norm $\opnorm{\cdot}$ (largest singular value)---$\opnorm{\Delta} \ll \opnorm{zz^*} = n$.
\begin{lemma}\label{lemma:ell2bound}
	If $x\in\Cn$ verifies $\|x\|_2^2 = n$ and $x^*Cx \geq z^*Cz$ (in particular, if $x$ is a global optimum of~\eqref{eq:P}), then
	\begin{align*}
        d(z, x) &
		\leq 4\frac{\opnorm{\Delta}}{\sqrt{n}}.
	\end{align*}
\end{lemma}
\begin{proof}
	See~\cite[Lemma 4.1]{bandeira2014tightness}.
\end{proof}
Under the white Gaussian noise scenario, the Cram\'er--Rao bound for phase synchronization~\cite{crbsynch,howard2010estimation} states that no unbiased estimator for $z$ based on $C$ can have expected squared error lower than $c\opnorm{\Delta}^2/n$, for some constant $c$. Thus, the maximum likelihood estimator (MLE) is order optimal. Similar information theoretic bounds applied to phase synchronization can be found in~\cite{javanmard2015phase}.

The latter result is motivation to compute global optima of~\eqref{eq:P}.\footnote{As shown in Section~\ref{sec:eig}, the simple \emph{eigenvector estimator} is statistically almost as good as the MLE. In this paper, we focus on the optimization problem~\eqref{eq:P} to obtain the actual MLE.} In general, this is NP-hard~\cite[Prop.\,3.5]{zhang2006complex}. In fact, even checking whether a candidate optimum is but a local optimum could be NP-hard in general, as is the case for nonconvex quadratic programming~\citep[\S5.1]{vavasis1991nonlinear}. Fortunately, for~\eqref{eq:P}, global optimality can sometimes be certified through the Hermitian matrix
\begin{align}
	S = S(x) = \Re\{\ddiag(Cxx^*)\} - C,
	\label{eq:S}
\end{align}
where $\ddiag \colon \Cnn \to \Cnn$ zeroes out all off-diagonal entries of a matrix~\citep{bandeira2014tightness}. This is captured in the following lemma.
\begin{lemma} \label{lemma:sufficientS}
	Let $x_\mathrm{opt}$ be globally optimal for~\eqref{eq:P}. For any $x\in\Cn_1$, the optimality gap at $x$ is bounded as
	\begin{align}
		0 \quad \leq \quad f(x_\mathrm{opt}) - f(x) \quad \leq \quad -n\lambdamin\big(S(x)\big).
		\label{eq:boundsvalueP}
	\end{align}
	In particular, if $S(x) \succeq 0$, then $x$ is globally optimal for~\eqref{eq:P}. If furthermore we have $\rank(S(x)) = n-1$, then the global optimum is unique (up to phase).
\end{lemma}
\begin{proof}
	Recall the definition of $f(x) = x^*Cx$. For all $y\in\Cn_1$, it holds that
	\begin{align*}
		n \cdot \lambdamin(S(x)) \leq y^*S(x)y = y^* \Re\{\ddiag(Cxx^*)\} y - y^* C y = x^*Cx - y^*Cy.
	\end{align*}
	Take $y=x_\mathrm{opt}$ to establish~\eqref{eq:boundsvalueP}. For uniqueness, see~\cite[\S4.2]{bandeira2014tightness}.
\end{proof}
Note that this optimality condition is only sufficient: in general, $S(x)$ may be indefinite even at a globally optimal $x$. The theorems in this paper identify and benefit from a regime where the condition is also necessary.

When $S$ is positive semidefinite, it is optimal for the Lagrangian dual of~\eqref{eq:P}, whose dual (the bi-dual of~\eqref{eq:P}) is the classical semidefinite relaxation of~\eqref{eq:P} obtained by lifting. Thus, $S$ as a certificate is most useful when that semidefinite relaxation is tight, which is the topic of~\cite{bandeira2014tightness}. When the semidefinite relaxation is tight, the global optimum of~\eqref{eq:P} can be computed in polynomial time by solving the associated semidefinite program. Unfortunately, in practice, this is typically slow because the relaxation involves lifting the problem to a (much) higher dimensional space ($O(n^2)$ compared to $O(n)$).

In this paper, we show that a simple \emph{generalized power method} (GPM, Algorithm~\ref{algo:ppm}, Section~\ref{sec:ppm}) in the low-dimensional space $\Cn_1$ converges to the global optimum of~\eqref{eq:P} (unique up to phase). This is under some conditions on $\Delta$ and provided GPM is adequately initialized, for example via the \emph{eigenvector method} (Section~\ref{sec:eig}). The method alternates between applying $C + \alpha I_n$ to the current iterate (note the extra inertia term $\alpha$) and projecting to the nonconvex set $\Cn_1$. The main result follows. (Setting $\alpha = \max(0, -\lambdamin(C))$---the smallest allowed value in Algorithm~\ref{algo:ppm}---is practical and satisfies the assumption below.)
\begin{theorem}\label{thm:ppmglobalopt}
	If $\opnorm{\Delta} \leq \frac{1}{28}n^{2/3}$ and $\|\Delta z\|_\infty \leq \frac{1}{28}n^{2/3}\sqrt{\log n}$, and if $\alpha \leq \opnorm{\Delta}$, then the iterates produced by Algorithm~\ref{algo:ppm} converge to
	the unique global optimum of~\eqref{eq:P}, considered in the quotient space $\Cn_1/\!\!\sim$ (that is, up to phase).

\end{theorem}
The bound on $\opnorm{\Delta}$ conservatively ensures the perturbation $\Delta$ cannot contain a competing signal $yy^*$ strong enough to overshadow $zz^*$. The bound on $\|\Delta z\|_\infty$ ensures no row of $\Delta$ aligns strongly with $z$, as otherwise $\Delta$ could reduce the signal strength by including a component aligned with $-zz^*$. If the perturbation is independent of the signal, the bound on $\|\Delta z\|_\infty$ is not expected to be the bottleneck. If the perturbation is a function of $z$, which may be the case in certain applications, a different analysis may be needed. For example, see~\citep{bandeira2012cheeger} for a treatment of adversarial noise.

In Section~\ref{sec:landscape}, we further explore the landscape of~\eqref{eq:P} independently of GPM, and find that, in the same noise regime, second-order necessary optimality conditions are also sufficient. Thus, regardless of initialization, any algorithm for~\eqref{eq:P} which converges to a point satisfying first- and second-order necessary optimality conditions converges to a global optimum. This is the case for the Riemannian trust-region method for example (RTR)~\citep{genrtr,boumal2016globalrates}. (The condition on the diagonal of $C$ is harmless, since changing the diagonal of $C$ only offsets the cost $f$ by a constant on $\Cn_1$.)
\begin{theorem}\label{thm:socpglobalopt}
	If $\diag(C) \geq 0$, $\opnorm{\Delta} \leq \frac{1}{14}n^{2/3}$ and $\|\Delta z\|_\infty \leq \frac{1}{14}n^{2/3}\sqrt{\log n}$, then all second-order critical points of~\eqref{eq:P} are global optima of~\eqref{eq:P} (unique up to phase).
\end{theorem}

The conditions on $\Delta$ are satisfied with high probability under the white Gaussian noise model, provided the variance is not too large.
\begin{lemma}\label{lemma:Wignernoise}
	If $\{\Delta_{ij}, i < j\}$ are i.i.d.\ complex Gaussian random variables with variance $\sigma^2$,
	$\Delta_{ii} = 0$ and $\Delta_{ji} \triangleq \overline \Delta_{ij}$ for $i < j$, then $\opnorm{\Delta} \leq 3\sigma \sqrt{n}$ and $\|\Delta z\|_\infty \leq 3\sigma \sqrt{n \log n}$ with probability at least $1 - 2n^{-5/4} - e^{-n/2}$.
\end{lemma}
\begin{proof}
	See \cite[Prop.\,3.3]{bandeira2014tightness}.
\end{proof}
In this Gaussian setting, a global optimum of~\eqref{eq:P} is an MLE for $z$, and the main theorems apply positively for $\sigma = O(n^{1/6})$. In comparison, numerical experiments suggest $\sigma = \tilde{O}(n^{1/2})$ is acceptable---see Section~\ref{sec:XP}. The bottleneck in the analysis is isolated in Lemma~\ref{lem:Deltaxinfty}.

After the first appearance of this paper and partly in response to it, Liu et al.~\citep{liu2016statistical} established the linear convergence rate of GPM, allowing $\sigma$ up to $O(n^{1/4})$. This matches the performance guarantee for the SDP relaxation in~\citep{bandeira2014tightness}.

\subsection*{Related work}

Quoting~\cite{bandeira2014tightness}, ``phase synchronization notably comes up in time-synchronization of distributed networks~\cite{giridhar2006distributed}, signal reconstruction from phaseless measurements~\cite{Alexeev_PhaseRetrievalPolarization,Bandeira_FourierMasks}, ranking~\cite{cucuringu2015syncrank}, digital communications~\cite{So_2010_SDP_detector}, and surface reconstruction problems in computer vision~\cite{agrawal2006surface} and optics~\cite{rubinstein2001optical}. Angular synchronization serves as a model for the more general problem of synchronization of rotations in any dimension, which comes up in structure from motion~\cite{matinec2007rotation,hartley2013rotation}, surface reconstruction from 3D scans~\cite{wang2012LUD} and cryo-electron microscopy~\cite{singer2011eigen}, to name a few.'' Problem~\eqref{eq:P} is also closely related to phase retrieval from short-time Fourier transforms~\citep{bendory2016stft} with applications in ultra-short pulse measurements and ptychography, and to unimodular codes~\citep{soltanalian2014designing}, where algorithms similar to GPM are also studied.

The eigenvector method (Section~\ref{sec:eig}) and the semidefinite relaxation method for synchronization already appear in a study by Singer~\cite{singer2010angular}. After computing either a dominant eigenvector of $C$ or a solution of the semidefinite relaxation of~\eqref{eq:P}, both methods project the relaxed solution to $\Cn_1$. Javanmard et al.~\cite{javanmard2015phase} recently study phase transitions of estimation quality versus noise level for both these estimators (up to the fact that they allow the estimators to not lie in $\Cn_1$ exactly, so that the projection is replaced by a careful scaling). They argue that such estimators compare well to information-theoretic limits, in precisely identified noise regimes. From their own account, the argument---based on tools from statistical mechanics---is non-rigorous. Yet, it produces a useful picture of the situation.

Bandeira et al.~\cite{bandeira2014tightness} consider the case of Gaussian noise more closely, and in particular establish that the semidefinite relaxation method is exact, assuming some bounds on the noise level (corresponding to $\sigma = O(n^{1/4})$ in the context of Lemma~\ref{lemma:Wignernoise}.) That is: the solution of the semidefinite relaxation requires no projection in that setting. Since semidefinite programs (SDP) can be solved in polynomial time, this shows~\eqref{eq:P} is not NP-hard in that noise regime (with high probability).

The present paper completes the picture by proposing to use GPM for the task---a more practical algorithm than SDP solvers---with a proof of convergence to global optima. GPM was introduced by Journ\'ee et al.~\cite{journee2010generalized} to solve problems of the form
\begin{align}
		\max_{x \in \calC} g(x),
		\label{eq:maxCg}
\end{align}
where $g$ is convex (not necessarily differentiable) and $\calC$ is compact. This formalism applies here because, even though $f$ in~\eqref{eq:P} is not convex in general, it can be shifted to a convex function using the inertia term.

GPM iteratively optimizes a linear approximation of $g$ around $x_k$ over $\calC$ to obtain the next iterate $x_{k+1}$. As such, it is a nonconvex instance of the Frank--Wolfe algorithm. Owing to convexity of $g$, this iteration ensures monotonic improvement of the cost---see Lemma~\ref{lem:monotone}. For $\calC$ the unit sphere and $g$ the Rayleigh quotient $x^* A x$, GPM specializes to the classical power method, hence the name. Journ\'ee et al.~\cite{journee2010generalized} analyze this method for sparse PCA applications, where $\calC$ is either a sphere or a Stiefel manifold (set of orthonormal bases). They prove convergence to critical points, but
do not guarantee the quality of the limit points.
The same method also appears in a similar context as~\cite{journee2010generalized} under the name \emph{conditional gradients}~\cite{luss2013conditional}. There too, convergence to fixed points is established, but there is no characterization of the quality of the limit points. In both papers, experiments show the usefulness of such methods on various applications. The present paper proves that global optimality can sometimes be achieved with GPM.

The case $g(x)= x^*Ax$ in~\eqref{eq:maxCg} is of particular interest, because it captures the problem of computing a dominant eigenvector of $A$ under additional constraints. These extra constraints typically make the problem hard. In recent work, Deshpande et al.~\cite{deshpande2014conePCA} study the case where $\calC$ is the intersection of a convex cone with a unit sphere (they call this \emph{cone-constrained PCA}), with Montanari and Richard investigating the important case where the cone is the nonnegative orthant~\cite{montanari2014nonnegativepca}. The algorithms they explore are either of the GPM type, or of the \emph{approximate message passing} (AMP) type. AMP algorithms involve additional memory terms in the iteration (which appears to be different in nature from inertia terms). We would argue GPM is conceptually simpler. The results in~\cite{deshpande2014conePCA}, for example, guarantee convergence to an estimator whose risk (expected error) is within a constant factor from that of the MLE (provided a decent initial guess is available). This is weaker than convergence to a global optimum, but appears to accurately capture the phase transitions in the noise regime. Lemma~\ref{lem:accumptsaregoodfixedpts} is our closest result to such a statement. Note that the simple eigenvector method also achieves the right rates up to constants (Lemma~\ref{lem:eigell2}).

GPM also bears some resemblance with~\cite[Alg.\,5]{hartley2013rotation}, where rotations in $\Rd, d \geq 2$ are synchronized (as opposed to only in $\reals^2$ here). That paper provides no convergence analysis for the algorithm, but observes it works well in practice.

In a different context, GPM is somewhat similar to the algorithm AltMinPhase for phase retrieval~\cite{netrapalli2015altminphase}. There, a signal $x\in\Cn$ is to be estimated based on measurements of the moduli of $A\transpose x$, for some sensing matrix $A$; however, the phases of $A\transpose x$ are unavailable. After initialization based on a dominant singular vector of a data matrix, their algorithm proceeds with an alternating minimization approach where one of the two steps resembles the map $T$ in GPM. They prove convergence to a global optimum under certain conditions (for a resampling variant of the method). Another approach to phase retrieval, PhaseCut~\cite{waldspurger2012phase}, is posed as~\eqref{eq:P}, with a different model for $C$.

If $x$ lives in $\Rn$ rather than $\Cn$, then the constraints $|x_i| = 1$ place $x$ in the discrete space~$\{\pm 1\}^n$, allowing to model such problems as Max-Cut. One possible relaxation of that problem is to allow $x$ to live in $\Cn_1$ anyway, which has the benefit of making the search space connected and smooth (but still nonconvex).
Thus, GPM on~\eqref{eq:P} can be used as a heuristic for Max-Cut, similarly to rank-2 relaxations~\cite{burer2002rank,bandeira2016lowrankmaxcut}.

Our second main result, Theorem~\ref{thm:socpglobalopt}, exhibits a regime in which not only does~\eqref{eq:P} not have spurious local optima, but also
necessary optimality conditions become sufficient. This explains why~\eqref{eq:P} (in that noise regime) belongs to a growing family of nonconvex problems for which nonconvex algorithms are known to converge to global optima. Among others, see recent literature about dictionary learning~\cite{sun2015complete}, phase retrieval~\cite{sun2015nonconvex,sun2016geometric}, tensor decomposition~\cite{ge2015escaping}, certain rank minimization problems~\cite{liu2015snig,zhao2015nonconvex} and a family of smooth semidefinite programs~\citep{boumal2016bmapproach}.

\subsection*{Notation}

$\Cn_1$ is defined by~\eqref{eq:Cn1}. The entry-wise complex conjugate of $u$ is $\bar u$; $u^*$ is the conjugate-transpose of $u$;
$\opnorm{M}$ is the operator norm of matrix $M$ (its largest singular value); $\|u\|_p$ is the $\ell_p$-norm of vector $u$; $\inner{u}{v} = \Re\{u^*v\}$ is a real inner product on $\Cn$.
The principal square root of $-1$ is written $i$, or $j$ if confusion with index $i$ is possible. The $i$th entry of vector $x$ is $x_i$, whereas $x_k$ is the $k$th iterate of Algorithm~\ref{algo:ppm}. Numerical subscripts ($x_0, x_1, \ldots$) denote iterates. 

\section{The eigenvector method} \label{sec:eig}

A particularly simple and efficient approach to~\eqref{eq:P} is to resort to a spectral relaxation~\cite{singer2010angular}. Consider~\eqref{eq:P} where the $n$ constraints $|x_1| = \cdots = |x_n| = 1$ are relaxed to the unique constraint $\|x\|_2^2 = n$. Then, a properly scaled leading eigenvector\footnote{The leading eigenvector is efficiently computed using the standard power method started from a random initial guess
as soon as $\opnorm{\Delta} < n$.} of $C$, call it $v$, is a global optimum. Lemma~\ref{lemma:ell2bound} implies that $v$ is close to the ground truth signal $z$~\eqref{eq:Cn1} (this also follows from standard perturbation analysis in matrix theory). While $v$ need not be in $\Cn_1$, fortunately, its projection to $\Cn_1$ is also close to $z$. Indeed, denote the \emph{eigenvector estimator} $\hat v$ as the component-wise projection of $v$ to $\Cn_1$:
\begin{align}
	\hat v_i & =
	\begin{cases}
		\frac{v_i}{|v_i|} & \textrm{ if } v_i \neq 0, \\
		\frac{e^* v}{|e^* v|} & \textrm{ otherwise,}
	\end{cases}
	\label{eq:eig}
\end{align}
where $e\in\Cn$ is an arbitrary vector such that $e^* v \neq 0$.\footnote{For $i$ such that $v_i=0$, one can set the phase $\hat v_i$ arbitrarily. The choice $\frac{e^* v}{|e^* v|}$ has the aesthetic advantage of preserving invariance under global phase shift: $u = v e^{i\theta} \implies \hat u = \hat v e^{i\theta}$.}
The following lemma exploits a result of Liu et al.~\citep{liu2016statistical} to show that the eigenvector estimator is almost as close to $z$ as global optima of~\eqref{eq:P} (compare with Lemma~\ref{lemma:ell2bound}.)
\begin{lemma}\label{lem:eigell2}
	The eigenvector estimator $\hat v \in \Cn_1$ is close to $z$ in the following sense:
	\begin{align*}
		d(z, \hat v)
		\leq 8\frac{\opnorm{\Delta}}{\sqrt{n}}.
	\end{align*}
\end{lemma}
\begin{proof}
	Without loss of generality, pick the scale and global phase of a dominant eigenvector $v$ of $C$ such that $\|v\|_2^2 = n$ and $v^*z = |v^*z|$. With $\hat v$ as in~\eqref{eq:eig}, we find that
	\begin{align*}
		d(z, \hat v) = \min_{\theta \in \reals}\|\hat v e^{i \theta}-z\|_2 & \leq \|\hat v - z\|_2 \leq 2 \|v - z\|_2 \leq 8\frac{\opnorm{\Delta}}{\sqrt{n}}.
	\end{align*}
	The first inequality follows from setting $\theta = 0$; the second inequality follows from~\citep[Lemma~2]{liu2016statistical} (for self-containment, see proof in Appendix~\ref{apdx:eigimprovement}); the third inequality follows from Lemma~\ref{lemma:ell2bound}, because $v^*z = |v^*z|$ ensures $\|v-z\|_2 = d(z, v)$.
\end{proof}
Note that this lemma still makes sense for the real case, where one estimates signs $z \in \{\pm 1\}^n$ rather than phases, based on noisy measurements of relative signs $zz\transpose$.

\section{The generalized power method} \label{sec:ppm}

The eigenvector method produces a good estimator, but it does not, in general, produce a global optimum of~\eqref{eq:P}. This is because the constraint $x\in\Cn_1$ is not acknowledged until the final projection step.
Consequently, we consider GPM (Algorithm~\ref{algo:ppm}.)
It essentially mimics the power method,
with two differences: (i) when applying $C$ to the current vector, some inertia is added, and (ii) at each iteration, the \emph{individual} entries of the current vector are normalized, as opposed to normalizing the vector as a whole.

\begin{algorithm}[h]
	\caption{GPM: Generalized power method for phase synchronization}
	\label{algo:ppm}
	\begin{algorithmic}[1]
		\State \textbf{Input:} $C$~\eqref{eq:C}
		\State Initialize with the eigenvector method $x_0 = \hat v$~\eqref{eq:eig} (not important in practice)
		\State Choose $\alpha \geq 0$ such that $C + \alpha I_n \succeq 0$ (e.g., $\alpha = \opnorm{\Delta}$ or $\alpha = \max(0, -\lambdamin(C))$)
		\State Define $\tilde C = C + \alpha I_n$
		\State Define $T \colon \Cn_1 \to \Cn_1,$ with
		\begin{align}
			T(x)_i & = \begin{cases} (\tilde Cx)_i / |(\tilde Cx)_i| & \textrm{if } (\tilde Cx)_i \neq 0, \\ x_i & \textrm{otherwise.} \end{cases}
			\label{eq:T}
		\end{align}
		\For {$k = 0, 1, 2 \ldots$ }
		\State $x_{k+1} = T(x_k)$
		\EndFor
	\end{algorithmic}
\end{algorithm}

Notice that GPM iterates on the quotient space $\Cn_1 /\!\! \sim$, that is: $T$~\eqref{eq:T} iterates from equivalence class to equivalence class~\eqref{eq:equivalence}. Formally, $T(xe^{i\theta}) = T(x)e^{i\theta}$, which implies $x \sim y \implies T(x) \sim T(y)$.

\begin{remark}
	
	GPM is
	a
	projected gradient method. Indeed, letting $P\colon \Cn\to\Cn_1$ denote entry-wise modulus normalization, and disregarding (unlikely) divisions by zero for now, $x_{k+1} = T(x_k) = P(\alpha x_k + \tilde C x_k) = P(x_k + \frac{1}{2\alpha} 2\tilde C x_k)$, where we recognize $2\tilde C x_k$ as the gradient of $x \mapsto x^*\tilde C x$. Unfortunately, $P$ can be expansive. This precludes using standard convergence analyses of projected gradient methods for GPM.
	Furthermore, in general, GPM is \emph{not} a Riemannian gradient descent~\cite{AMS08}. Indeed, for certain valid choices of $\alpha$, there may exist critical points (see~\eqref{eq:criticalcondition}) which are not fixed for $T$: let $\tilde C = zz^* + \alpha I_n$ such that $\tilde C$ is strictly not diagonally dominant; then, no suboptimal critical point is fixed for $T$. This precludes using these standard convergence analyses as well.

\end{remark}

We first characterize the \emph{fixed points} of $T$~\eqref{eq:T}, that is, points $x$ such that $T(x) = x$.
\begin{lemma} \label{lem:fixedpointl1norm}
	The following statements about $x\in\Cn_1$ are equivalent:
	\begin{enumerate}
		\item[(a)] $x$ is a fixed point of $T$ ;
		\item[(b)] $\forall i, (\tilde Cx)_i \bar x_i = |(\tilde Cx)_i|$ ;
		\item[(c)] $x^*\tilde Cx = \|\tilde Cx\|_1$.
	\end{enumerate}
\end{lemma}
\begin{proof}
	Equivalence of $(a)$ and $(b)$ is an easy consequence of the definition of $T$ ; $(b)$ implies $(c)$ by summing over $i = 1, \ldots, n$. We show $(c)$ implies $(b)$. If $\|\tilde Cx\|_1 = x^*\tilde Cx$, then, with the real inner product $\inner{a}{b} = \Re\{\bar a b\}$ over $\mathbb{C}$,
	\begin{align*}
		\sum_{i=1}^{n} |(\tilde Cx)_i| = \sum_{i=1}^{n} \innersmall{x_i}{(\tilde Cx)_i}.
	\end{align*}
	The terms of the sum on the left hand side dominate those of the sum on the right hand side term by term (by Cauchy--Schwarz and $|x_i|=1$). Thus, the equality of the sums requires equality of all terms: $|(\tilde Cx)_i| = \innersmall{x_i}{(\tilde Cx)_i}$ for all $i$. This implies $(b)$.

\end{proof}

We now show GPM enjoys monotonic cost improvement along iterates, owing to the inertia term $\alpha$. 
The proof relies on the fact that GPM is a nonconvex instance of the Frank--Wolfe algorithm, as previously exploited in~\cite{journee2010generalized,luss2013conditional}.
As shown later, another effect of the inertia term is to ensure global optima of~\eqref{eq:P} are fixed points of the algorithm---see Lemma~\ref{lem:linkfixedsocp}.
\begin{lemma} \label{lem:monotone}
	Iterates produced by Algorithm~\ref{algo:ppm} satisfy $f(x_{k+1}) > f(x_{k})$ (with $f$ the cost function of~\eqref{eq:P}), unless $x_k = x_{k+1}$. In particular, $(f(x_k))_{k=0,1,\ldots}$ converges to a finite value and the iterates $x_0, x_1, \ldots$ do not cycle.
\end{lemma}
\begin{proof}
	Let $g(x) = x^* \tilde C x$, recalling that $\tilde C = C + \alpha I_n$. Since $f-g$ is constant on $\Cn_1$, showing monotonous improvement of $g(x_k)$ implies the same for $f(x_k)$. Once this is shown, convergence of $(f(x_k))_{k=0,1,\ldots}$ will follow directly from compactness of $\Cn_1$ and continuity of $f$. Thus, we concentrate on showing monotonous improvement of $g(x_k)$.
	
	Owing to $\tilde C \succeq 0$, $g$ is a convex function.

	Hence, the following inequality holds for all $y\in\Cn$:
	\begin{align}
	g(y) - g(x_k) \geq 2\innersmall{y-x_k}{\tilde Cx_k}.
	\label{eq:convexityg}
	\end{align}
	Setting $y = T(x_k)$ maximizes the right hand side for $y \in \Cn_1$ (it is in that sense that Algorithm~\ref{algo:ppm} is nothing but a nonconvex instance of the Frank--Wolfe algorithm):
	\begin{align*}
	x_{k+1} = T(x_k) \in \arg\max_{y\in\Cn_1} \, \innersmall{y}{\tilde Cx_k}.
	\end{align*}
	To see this, use that both the cost and the constraints are separable for each entry of $y$, then use Cauchy--Schwarz; for $i$ such that $(\tilde Cx_k)_i = 0$, any unit-modulus complex number will do for $y_i$. Since setting $y = x_k$ would yield a zero improvement, the right hand side of~\eqref{eq:convexityg} with $y = x_{k+1}$ is nonnegative: $g(x_{k+1}) \geq g(x_{k})$.
	
	We now show that the improvement in $g$ is strict if $x_{k+1} \neq x_k$. To this end, notice that $\innersmall{x_{k+1}}{\tilde Cx_k} = \|\tilde Cx_k\|_1$ since $x_{k+1}$ contains the phases of the entries of $\tilde C x_k$. Thus, using nonnegativity again,
	\begin{align*}
		0 \leq \innersmall{x_{k+1} - x_k}{\tilde C x_k} = \innersmall{x_{k+1}}{\tilde C x_k} - \innersmall{x_k}{\tilde C x_k} = \|\tilde C x_k\|_1 - g(x_k),
	\end{align*}

	so it holds that $g(x_k) \leq \|\tilde Cx_k\|_1$. Using this and~\eqref{eq:convexityg}, it also follows that
	\begin{align*}
		g(x_{k+1}) & \geq g(x_k) + 2\left( \|\tilde C x_k\|_1 - g(x_k)\right)
		\geq \|\tilde C x_k\|_1.
	\end{align*}
	(Simply replace 2 by 1.) Hence, $g(x_{k+1}) \geq \|\tilde Cx_k\|_1 \geq g(x_k)$ so that the sequence $(\|\tilde Cx_k\|_1)_{k=0,1,\ldots}$ also grows monotonically. If $g(x_{k+1}) = g(x_{k})$, then $g(x_{k}) = \|\tilde Cx_k\|_1$.
	By Lemma~\ref{lem:fixedpointl1norm}, this implies
	$x_k = x_{k+1}$.

\end{proof}
As a consequence of monotonous improvement of $f$ along iterates, if $x$ is a fixed point of $T$ but it is not a local optimum of~\eqref{eq:P}, then it is not asymptotically stable. That is, there exist points $y\in\Cn_1$ arbitrarily close to $x$ such that $d(x, T^k(y))$ does not converge to 0 with $k\to\infty$. This makes convergence to such fixed points unlikely in practice.

The two previous lemmas lead to a weak convergence result for the iterates, akin to the classical convergence results for gradient descent methods in nonlinear optimization.
\begin{lemma}\label{lem:ppmaccumpts}
	All accumulation points of $(x_k)_{k=0,1,\ldots}$ generated by Algorithm~\ref{algo:ppm} are fixed points of $T$. If $x$ is an accumulation point (at least one exists), then $f(x) \geq f(x_0)$.
\end{lemma}
\begin{proof}
	Continuing the proof of Lemma~\ref{lem:monotone} (specifically, eq.~\eqref{eq:convexityg}), we have (with $g(x) = x^* \tilde C x$)
	\begin{align*}
		\frac{1}{2}\left[g(x_{k+1}) - g(x_k)\right] \geq \innersmall{x_{k+1} - x_k}{\tilde Cx_k} = \|\tilde Cx_k\|_1 - g(x_k) \geq 0.
	\end{align*}
	Summing for $k = 0, \ldots, K-1$, we get, for all $K \geq 1$,
	\begin{align*}
		\frac{1}{2}\left[g(x_{K}) - g(x_0)\right] \geq \sum_{k = 0}^{K-1} \|\tilde Cx_k\|_1 - g(x_k).
	\end{align*}
	Since $(g(x_k))_{k=0,1,\ldots}$ converges (say, to $g_\infty < \infty$), the (nonnegative) terms of the sum must converge to zero:
	\begin{align*}
		\lim_{k\to\infty} \|\tilde Cx_k\|_1 - g(x_k) = 0.
	\end{align*}
	Since $\Cn_1$ is compact, there is a convergent subsequence of iterates whose limit point $x$ lies in $\Cn_1$ (Bolzano--Weierstrass). By continuity, this limit point verifies $\|\tilde Cx\|_1 - g(x) = 0$, hence it is a fixed point of $T$ (Lemma~\ref{lem:fixedpointl1norm}). Furthermore, since the cost function $f$ is continuous and $f(x_{k+1}) \geq f(x_k)$ (Lemma~\ref{lem:monotone}), $$f(x) = \sup\{ f(x_k) : k \textrm{ indexes the subsequence} \} \geq f(x_0).$$ This is true of all accumulation points of $(x_k)_{k=0,1,\ldots}$.

\end{proof}
(Slightly increasing $\alpha$ in Algorithm~\ref{algo:ppm} further ensures that the set of accumulation points is connected, see Lemma~\ref{lem:connectedness} in Appendix~\ref{apdx:connectedness}.)

We now show that, provided the noise $\Delta$ is not too large, fixed points of $T$ are either far from the signal $z$, or close to it.
\begin{lemma} \label{lem:fixedpointsgoodbad}
	If $x$ is a fixed point for $T$~\eqref{eq:T}, (at least) one of the following holds:
	\begin{align*}
	|z^*x|  \geq n - 4(\opnorm{\Delta} + \alpha), & & \textrm{ or } & & |z^*x| \leq 4(\opnorm{\Delta} + \alpha).
	\end{align*}
	Exactly one statement holds if $\opnorm{\Delta} + \alpha < n/8$.
\end{lemma}
\begin{proof}
	The proof compares upper and lower bounds on $x^*\tilde Cx$. On one hand, using $\tilde C = zz^* + \Delta + \alpha I_n$, we have
	$$
	x^*\tilde Cx = |z^*x|^2 + x^*\Delta x + \alpha n \leq |z^*x|^2 + n\opnorm{\Delta} + \alpha n.
	$$
	On the other hand, since $x$ is a fixed point, Lemma~\ref{lem:fixedpointl1norm} implies
	$$
	x^*\tilde Cx = \| \tilde Cx \|_1 = \sum_{i=1}^n \Big| (z^*x)z_i + (\Delta x)_i + \alpha x_i \Big| \geq n|z^*x| - \|\Delta x\|_1 - \alpha n.
	$$
	Combining both inequalities and $\|\Delta x\|_1 \leq \sqrt{n}\|\Delta x\|_2 \leq n \opnorm{\Delta}$, we find that
	\begin{align}
	|z^*x|(n - |z^*x|) \leq 2n(\opnorm{\Delta} + \alpha).
	\label{eq:quadraticinequality}
	\end{align}
	Since $0 \leq |z^*x| \leq n$, the left hand side is at most $n^2/4$. Thus, the inequality is informative only if $\opnorm{\Delta} + \alpha < n/8$, which we assume to hold (otherwise, the lemma's statement is trivial).
	
	Consider the roots $r_\pm$ of the concave quadratic $q(t) = t(n - t) - 2n(\opnorm{\Delta} + \alpha)$:
	\begin{align*}
		r_\pm = n \frac{1 \pm \sqrt{1 - \frac{8(\opnorm{\Delta} + \alpha)}{n}}}{2}.
	\end{align*}
	By~\eqref{eq:quadraticinequality}, we know that $q(|z^*x|) \leq 0$, that is, $|z^*x|$ is in $[0, r_-]$ or $[r_+, n]$. Then, using twice that $\sqrt{1 - s} \geq 1-s$ for all $s \in [0, 1]$, it follows that either
	\begin{align*}
	|z^*x| & \geq r_+ =  n \frac{1 + \sqrt{1 - \frac{8(\opnorm{\Delta} + \alpha)}{n}}}{2} \geq n\left(1 - \frac{4(\opnorm{\Delta} + \alpha)}{n}\right), \textrm{ or } \\
	|z^*x| & \leq r_- = n \frac{1 - \sqrt{1 - \frac{8(\opnorm{\Delta} + \alpha)}{n}}}{2} \leq 4(\opnorm{\Delta} + \alpha).
	\end{align*}
	This concludes the proof.
\end{proof}
Since $\alpha = \opnorm{\Delta}$ is always an acceptable choice, this last result states that, as soon as $\opnorm{\Delta} < n/16$, the fixed points of $T$ are separated in two sets: one concentrated around $z$ (the ``good'' fixed points), and one bounded away from $z$ (the ``bad'' fixed points). The latter must have poor performance in terms of the cost function $f$ as well. Then, owing to monotonicity, a good initialization is sufficient to rule out convergence to the bad fixed points. This is formalized in a lemma.
\begin{lemma} \label{lem:accumptsaregoodfixedpts}
	If $\opnorm{\Delta} \leq n/13$ and if $\alpha \leq \opnorm{\Delta}$, all accumulation points $x$ of Algorithm~\ref{algo:ppm} are good fixed points, that is, they satisfy $|z^*x| \geq n - 8\opnorm{\Delta}$.
\end{lemma}
\begin{proof}
	It is sufficient to show that $x_0 = \hat v$ (the eigenvector estimator from Section~\ref{sec:eig}) outperforms all bad fixed points in terms of the cost $x^*Cx$. From Lemma~\ref{lem:eigell2}, we have (fixing the phase $z^*x_0 = |z^*x_0|$ without loss of generality)
	\begin{align*}
		2(n - |z^*x_0|) = \|x_0 - z\|_2^2 \leq 64\frac{\opnorm{\Delta}^2}{n}.
	\end{align*}
	This allows to bound the performance of $x_0$ in terms of the cost $f$:
	\begin{align*}
		x_0^*Cx_0^{} & = |z^*x_0|^2 + x_0^* \Delta x_0^{} \\
			& \geq \left( n - 32\frac{\opnorm{\Delta}^2}{n} \right)^2 - n\opnorm{\Delta} \\
			& \geq n^2 - 64\opnorm{\Delta}^2 - n\opnorm{\Delta}.
	\end{align*}
	Likewise, we know from Lemma~\ref{lem:fixedpointsgoodbad} that bad fixed points $y$ are far from $z$: $|z^*y| \leq 4(\opnorm{\Delta} + \alpha) \leq 8\opnorm{\Delta}$. Hence,
	\begin{align*}
		y^*Cy & = |z^*y|^2 + y^*\Delta y 
			\leq 64\opnorm{\Delta}^2 + n\opnorm{\Delta}.
	\end{align*}
	Hence, if
	\begin{align}
		n^2 - 64\opnorm{\Delta}^2 - n\opnorm{\Delta} > 64\opnorm{\Delta}^2 + n\opnorm{\Delta},
		\label{eq:eigperffoo}
	\end{align}
	the initial point $x_0$ outperforms all bad fixed points. Then, by Lemma~\ref{lem:ppmaccumpts}, all accumulation points of Algorithm~\ref{algo:ppm} are good fixed points. We show a sufficient condition for~\eqref{eq:eigperffoo}~to hold. Assuming $\opnorm{\Delta} \leq cn$ for some constant $c$ to be determined, condition~\eqref{eq:eigperffoo} is satisfied in particular if (divide through by $n^2$)
	\begin{align*}
		1 - 128c^2 - 2c > 0. 
	\end{align*}
	This is satisfied for $c = 1/13$.
\end{proof}
We are almost ready to prove Theorem~\ref{thm:ppmglobalopt}. The next technical lemma is the last piece we need. It is presented independently because it is suboptimal and constitutes the bottleneck in our analysis.
\begin{lemma} \label{lem:Deltaxinfty}
	If $\opnorm{\Delta} \leq cn$ and $\|\Delta z\|_\infty \leq cn\sqrt{\log n}$ for some $c$, then $\|\Delta x\|_\infty \leq (\sqrt{\log n} + \|x-z\|_2 ) cn$.
\end{lemma}
\begin{proof}
	The proof starts with a triangular inequality:
	\begin{align*}
	\|\Delta x\|_\infty & \leq \|\Delta z\|_\infty + \|\Delta (x-z)\|_\infty \\
	& \leq cn\sqrt{\log n} + \|\Delta(x-z)\|_2 \\
	& \leq (\sqrt{\log n} + \|x-z\|_2 ) cn.
	\end{align*}
	The suboptimal step occurs when bounding an $\ell_\infty$-norm with an $\ell_2$-norm.
\end{proof}
We now prove the main result about GPM. This part of the proof draws heavily on~\cite[\S4.4]{bandeira2014tightness}.
\begin{proof}[Proof of Theorem~\ref{thm:ppmglobalopt}]
	By Lemma~\ref{lem:accumptsaregoodfixedpts}, all accumulation points $x$ are fixed points satisfying $|z^*x| \geq n - 8\opnorm{\Delta}$.
	We are about to show that, for such $x$, $S = S(x)$~\eqref{eq:S} is positive semidefinite and has rank $n-1$. By Lemma~\ref{lemma:sufficientS}, this proves all accumulation points are globally optimal, and that the global optimum is unique up to phase. For convenience, we work with the equivalent definition of $S$ at fixed points:
	\begin{align*}
		S = \ddiag(\tilde Cxx^*) - \tilde C,
	\end{align*}
	where $\tilde C$ appears instead of $C$ because $S$ is invariant under diagonal shifts, and where there is no need to extract the real part of $\diag(\tilde Cxx^*)$ since, for fixed points, this vector is real, nonnegative (Lemma~\ref{lem:fixedpointl1norm}).
	
	Observe that $Sx = 0$. Hence, it suffices to show that $u^*Su > 0$ for all nonzero $u\in\Cn$ such that $u^*x = 0$. Without loss of generality, assume $z^*x = |z^*x|$. Using Lemma~\ref{lem:fixedpointl1norm} which says $(\tilde C x)_i \bar x_i = |(\tilde C x)_i|$ and $\tilde C = zz^* + \Delta + \alpha I_n$, we have

	\begin{align*}
		u^*Su & = \sum_{i=1}^{n} |u_i|^2 |(\tilde Cx)_i| - u^*\tilde Cu \\
		& = \sum_{i=1}^{n} |u_i|^2 \big| |z^*x| z_i + (\Delta x)_i + \alpha x_i\big| - |u^*z|^2 - u^*\Delta u - \alpha \|u\|_2^2 \\
		& \geq \sum_{i=1}^{n} |u_i|^2 \left( |z^*x| - |(\Delta x)_i| - \alpha \right) - |u^*(z-x)|^2 - u^*\Delta u  - \alpha \|u\|_2^2\\
		& \geq \|u\|_2^2 \left( |z^* x| - \|\Delta x\|_\infty - \|z-x\|_2^2 - \opnorm{\Delta} -2\alpha\right) \\
		& \geq \|u\|_2^2 \left( n - 27\opnorm{\Delta} - \|\Delta x\|_\infty \right),
	\end{align*}
	where we use $|z^*x| \geq n - 8\opnorm{\Delta}$ twice in the last inequality. Assuming $\opnorm{\Delta} \leq cn$ and $\|\Delta z\|_\infty \leq cn\sqrt{\log n}$ for some $c$ to be determined, and using Lemma~\ref{lem:Deltaxinfty}, it comes that a sufficient condition for the latter to be positive is
	\begin{align*}
		n - (27 + \sqrt{\log n} + 4\sqrt{cn})cn > 0.
	\end{align*}
	Further assuming $c = c' n^{-1/3}$, and observing that $n^{-1/3}(27 + \sqrt{\log n}) \leq 27$ for all $n \geq 1$, a sufficient condition is
	\begin{align*}
		1 - (27 + 4 \sqrt{c'} )c' > 0.
	\end{align*}
	This is satisfied for $c' = 1/28$.
	
	So far, we showed all accumulation points of $(x_k)_{k=0,1,\ldots}$ (there exists one by compactness) are global optima of~\eqref{eq:P}, and that this global optimum is unique up to phase. Recall that $[x]$ denotes the equivalence class of $x\in\Cn_1$ for the equivalence relation~$\sim$~\eqref{eq:equivalence} and that $\Cn_1/\!\!\sim$ is the corresponding quotient space (the set of equivalence classes). The cost function $h([x]) = f(x) = x^*Cx$ is well defined and continuous on $\Cn_1/\!\!\sim$, which is a nonempty compact metric space. Let $x_\textrm{opt}$ be a global optimum of~\eqref{eq:P}. By continuity of $h$, it holds that $\lim_{k\to\infty} h([x_k]) = h([x_\textrm{opt}])$ (simply extract a convergent subsequence of $[x_0], [x_1], \ldots$, then switch $h$ and the limit).
	It then follows from technical Lemma~\ref{lem:cvgcelemma} (see appendix) that $([x_k])_{k=0,1,\ldots}$ converges to $[x_\textrm{opt}]$.

\end{proof}

In the regime of Theorem~\ref{thm:ppmglobalopt}, the bad fixed points cannot be local optima (see the proof of Lemma~\ref{lem:SOCPclosetoz} below). As a result, bad fixed points are not asymptotically stable (see the comment after Lemma~\ref{lem:monotone}).
This further explains why, typically, GPM converges to global optima even when $x_0$ is chosen at random (in this regime). Lemmas~\ref{lem:fixedpointl1norm}, \ref{lem:monotone}, \ref{lem:ppmaccumpts}, \ref{lem:fixedpointsgoodbad} and \ref{lem:connectedness} apply regardless of $x_0$.

We note in passing that GPM can be implemented in a decentralized fashion, akin to work by Howard et al.~\cite{howard2010estimation}.

\section{About sufficiency of necessary optimality conditions} \label{sec:landscape}

If $x$ is a global optimum of~\eqref{eq:P}, it satisfies first- and second-order necessary optimality conditions. All points who do are called \emph{second-order critical points}. In this section, we prove they are close to the signal $z$. Furthermore, we establish Theorem~\ref{thm:socpglobalopt} which shows the necessary conditions turn out to be sufficient for optimality. Both statements require assumptions on the perturbation $\Delta$ similar to those previously studied.

Sufficiency of the necessary conditions is an unusual and highly desirable property for nonconvex optimization. In particular, this shows that~\eqref{eq:P} can be solved by any algorithm which converges to second-order critical points. This is notably the case for certain versions of the Riemannian trust-region method~\cite{genrtr,boumal2016globalrates}---see Section~\ref{sec:XP}.

\begin{lemma} \label{lem:necessaryconditions}
	If  $x\in\Cn_1$ is a global optimum for~\eqref{eq:P}, then it satisfies first- and second-order necessary optimality conditions. First-order conditions require:
	\begin{align}
		S(x)x = 0,
		\label{eq:criticalcondition}
	\end{align}
	where $S(x)$ is defined by~\eqref{eq:S}. Second-order conditions require:
	\begin{align}
		\forall \dot x \in \T_x\Cn_1, \quad \inner{\dot x}{S(x)\dot x} \geq 0,
		\label{eq:socpcondition}
	\end{align}
	where
	\begin{align}
		\T_x\Cn_1 = \big\{ \dot x \in \Cn : \forall i, \inner{\dot x_i}{x_i} = \Re\{\dot x_i \bar x_i\} = 0 \big\}
		\label{eq:tangentspace}
	\end{align}
	is the tangent space at $x$ to the manifold $\Cn_1$.
	These conditions correspond, respectively, to having the Riemannian gradient of $f$ on $\Cn_1$ equal to zero \emph{(critical point)}, and the Riemannian Hessian of $f$ on $\Cn_1$ be negative semidefinite.
\end{lemma}
\begin{proof}
	See~\cite[\S4.3]{bandeira2014tightness}. See~\cite[eq.~(3.36),~(5.15)]{AMS08} for definitions of the Riemannian gradient and Hessian.
\end{proof}
We single out a few properties of these candidate optima.
\begin{lemma} \label{lem:critsocpproperties}
	$x\in\Cn_1$ is a critical point for~\eqref{eq:P} if and only if $\diag(Cxx^*)$ is real.  If $\diag(C) \geq 0$ and $x$ is a second-order critical point, then $\diag(Cxx^*) \geq 0$ and $x^*Cx = \|Cx\|_1$.
\end{lemma}
\begin{proof}
	The first statement follows from the definition of $S$~\eqref{eq:S}:
	\begin{align*}
		Sx = 0 \iff \forall i, \Re\{(Cx)_i \bar x_i\}x_i = (Cx)_i \iff \forall i, \Re\{(Cx)_i \bar x_i\} = (Cx)_i \bar x_i.
	\end{align*}
	(To establish the second equivalence, multiply by $\bar x_i$ on both sides and use $|x_i|=1$.)
	For the second statement, assume $x$ is a second-order critical point. Then, for any canonical basis vector $e_i \in \Rn$, this is a tangent vector: $\dot x = (jx_i)e_i$.
	Following~\eqref{eq:socpcondition},
	\begin{align*}
		0 \leq \dot x^*S\dot x = |jx_i|^2 \cdot  e_i^* S e_i = S_{ii} = (Cx)_i \bar x_i - C_{ii}.
	\end{align*}
	Thus, $\diag(Cxx^*) \geq \diag(C) \geq 0$. Consequently, $(Cx)_i \bar x_i = |(Cx)_i|$ for all $i$.
\end{proof}
As an aside, we note a strong link between the fixed points of GPM (who depend on $\alpha$) and the critical points of~\eqref{eq:P} (who do not).
\begin{lemma} \label{lem:linkfixedsocp}
	All fixed points of $T$~\eqref{eq:T} are critical points of~\eqref{eq:P}.
	If $\diag(\tilde C) \geq 0$ (as ensured in Algorithm~\ref{algo:ppm} by forcing $\tilde C \succeq 0$), then all second-order critical points of~\eqref{eq:P} are fixed points of $T$.
\end{lemma}
\begin{proof}
	The first statement is clear:
	\begin{align*}
		x = T(x) \stackrel{(Lemma~\ref{lem:fixedpointl1norm})}{\implies} \diag(\tilde Cxx^*) \textrm{ is real} \iff \diag(Cxx^*) \textrm{ is real} \stackrel{(Lemma~\ref{lem:critsocpproperties})}{\iff} Sx = 0.
	\end{align*}
	For the second statement, consider the second argument in the proof of Lemma~\ref{lem:critsocpproperties} with the equivalent definition of $S$ using $\tilde C$: $S = \ddiag(\tilde Cxx^*) - \tilde C$ and $x$ second-order critical. The argument states $0 \leq \diag(S) = \diag(\tilde Cxx^*) - \diag(\tilde C)$. Assuming $\diag(\tilde C) \geq 0$, this shows $\diag(\tilde Cxx^*) \geq 0$, which by Lemma~\ref{lem:fixedpointl1norm} implies $x$ is a fixed point.
\end{proof}

All second-order critical points of~\eqref{eq:P} are close to the signal $z$.
\begin{lemma}\label{lem:SOCPclosetoz}
	If $\diag(C) \geq 0$ and $\opnorm{\Delta} \leq n/13$, then all second-order critical points $x$ of~\eqref{eq:P} obey
	\begin{align*}
		|z^*x| \geq n - 4\opnorm{\Delta}.
	\end{align*}
	Thus, up to phase, they are all close to $z$ as $d(z, x)^2 \leq 8 \opnorm{\Delta}.$
\end{lemma}

In order to prove this lemma, we first need a technical result akin to Lemma~\ref{lem:fixedpointsgoodbad}. We show all second-order critical points are either close to $z$, or far away from $z$. We will then conclude by showing there can be no second-order critical points far away from $z$.

\begin{lemma} \label{lem:socpgoodbad}
	If $\diag(C) \geq 0$ and $x$ is a second-order critical point for~\eqref{eq:P}, (at least) one of the following holds:
	\begin{align*}
	|z^*x|  \geq n - 4\opnorm{\Delta}, & & \textrm{ or } & & |z^*x| \leq 4\opnorm{\Delta}.
	\end{align*}
	Exactly one statement holds if $\opnorm{\Delta} < n/8$.
\end{lemma}
\begin{proof}
	The proof is identical to that of Lemma~\ref{lem:fixedpointsgoodbad}, with $C$ instead of $\tilde C$ ($\alpha = 0$). The key property is $x^*Cx = \|Cx\|_1$, provided by Lemma~\ref{lem:critsocpproperties}.
\end{proof}

\begin{proof}[Proof of Lemma~\ref{lem:SOCPclosetoz}]
	Let $x\in\Cn_1$ be a second-order critical point of~\eqref{eq:P} such that $|z^*x| \leq 4 \opnorm{\Delta}$. We show such a point does not exist.
	
	Condition~\eqref{eq:socpcondition} holds for $x$. Consider tangent vectors $\dot x \in \T_x\Cn_1$~\eqref{eq:tangentspace} of the form $\dot x_i = s_i \cdot (jx_i)$, where $s_i = \pm 1$ for all $i$ (we will determine these signs momentarily).

	Let $\odot$ denote entry-wise multiplication in $\dot x = s \odot (jx)$, with $s \in \{\pm 1\}^n$. By direct computation,
	\begin{align}
		0 \leq \dot x^* S \dot x & = x^*Cx - (s \odot x)^*C(s \odot x) \nonumber\\
			& = |z^*x|^2 + x^*\Delta x - (s\odot x)^*\Delta(s\odot x) - |z^*(s\odot x)|^2 \nonumber\\
			& \leq 16 \opnorm{\Delta}^2 + 2n\opnorm{\Delta} - |z^*(s\odot x)|^2.
			\label{eq:xdotSxdotnegative}
	\end{align}
	(Notice that the first line shows a second-order critical point $x$ outperforms any point $s \odot x \in \Cn_1$ in terms of $f$, which is a kind of local-to-global statement.)
	The goal is to show the signs $s$ can be chosen such that $\dot x^* S \dot x < 0$, to reach a contradiction. To this end, consider:
	\begin{align*}
		\max_{s\in\{ \pm 1 \}^n} |z^*(s\odot x)|^2 & = \max_{s\in\{ \pm 1 \}^n} \Big| \sum_i s_i \bar z_i x_i \Big|^2 \\
		& = \max_{s\in\{ \pm 1 \}^n} \Re\Big\{ \sum_i s_i \bar z_i x_i \Big\}^2 + \Im\Big\{ \sum_i s_i \bar z_i x_i \Big\}^2 \\
		& = \max_{s\in\{ \pm 1 \}^n} \Big( \sum_i s_i \cos\theta_i \Big)^2 + \Big( \sum_i s_i \sin \theta_i  \Big)^2,
	\end{align*}
	where $e^{j\theta_i} := \bar z_i x_i$. Since both terms in the max are nonnegative, it further holds that
	\begin{align*}
	\max_{s\in\{ \pm 1 \}^n} |z^*(s\odot x)|^2 & \geq \max\left( \max_{s\in\{ \pm 1 \}^n} \Big( \sum_i s_i \cos\theta_i \Big)^2 , \max_{s\in\{ \pm 1 \}^n} \Big( \sum_i s_i \sin \theta_i  \Big)^2 \right) \\
	& = \max \left( \sum_i |\cos \theta_i|, \sum_i |\sin \theta_i| \right)^2 \\
	& \geq \left( \frac{1}{2}\sum_i |\cos \theta_i| + |\sin \theta_i| \right)^2 \geq \frac{1}{4}n^2,
	\end{align*}
	where we used $\max(a, b) \geq \frac{a+b}{2}$ for all $a, b\in\reals$ and $|\cos \theta| + |\sin \theta| \geq 1$ for all $\theta\in\reals$. (This last result says that, given any $x\in\Cn_1$, it is possible to pick signs $s$ such that $s \odot x$ correlates with $z$.) Plugging the optimal $s$ in~\eqref{eq:xdotSxdotnegative}, we find that if
	\begin{align}
		16 \opnorm{\Delta}^2 + 2n\opnorm{\Delta} < \frac{1}{4}n^2,
		\label{eq:intermediate3453}
	\end{align}
	then we reached a contradiction and $x$ is not second-order critical. Let $\opnorm{\Delta} \leq cn$ for some $c \geq 0$ to be determined. Eq.~\eqref{eq:intermediate3453} holds if $64c^2 + 8c < 1$. This holds in particular if $c \leq 1/13$, concluding the proof.
\end{proof}

The collected results of this section allow to prove the main theorem about equivalence of second-order critical points and optima of~\eqref{eq:P} (under the proposed regime for $\Delta$.) The cycle of implications is: $x$ optimal $\stackrel{\textrm{Lemma~\ref{lem:necessaryconditions}}}{\implies}$ $x$ second-order critical $\stackrel{\textrm{see below}}{\implies}$ $S(x) \succeq 0$ $\stackrel{\textrm{Lemma~\ref{lemma:sufficientS}}}{\implies}$ $x$ optimal.

\begin{proof}[Proof of Theorem~\ref{thm:socpglobalopt}]
	The proof is essentially that of Theorem~\ref{thm:ppmglobalopt}. Under the assumptions, a second-order critical point $x$ satisfies $Sx = 0$, $(Cx)_i\bar x_i = |(Cx)_i|$ (Lemma~\ref{lem:critsocpproperties}) and $|z^*x| \geq n - 4\opnorm{\Delta}$ (Lemma~\ref{lem:SOCPclosetoz})---we aim to show $S(x)\succeq 0$. For all $u\in\Cn$ such that $u^*x = 0$ (without loss of generality, assume $z^*x = |z^*x|$),
	\begin{align*}
	u^*Su & = \sum_{i=1}^{n} |u_i|^2 |(Cx)_i| - u^*Cu \\
	& \geq \|u\|_2^2 \left( |z^* x| - \|\Delta x\|_\infty - \|z-x\|_2^2 - \opnorm{\Delta} \right) \\
	& \geq \|u\|_2^2 \left( n - 13\opnorm{\Delta} - \|\Delta x\|_\infty \right).
	\end{align*}
	Assume $\opnorm{\Delta} \leq cn$ and $\|\Delta z\|_\infty \leq cn\sqrt{\log n}$ for some $c$. By Lemma~\ref{lem:Deltaxinfty}, $\|\Delta x\|_\infty \leq (\sqrt{\log n} + \sqrt{8cn} ) cn$. By Lemma~\ref{lemma:sufficientS}, a sufficient condition to establish that all second-order critical points are globally optimal, and uniqueness of the global optimum up to phase, becomes $1 - (13 + \sqrt{\log n} + \sqrt{8cn})c > 0$. If $c = c'n^{-1/3}$, since $n^{-1/3}(13+\sqrt{\log n}) \leq 13$, a sufficient condition is $1 - (13 + \sqrt{8c'})c' > 0$. This is satisfied for $c' = 1/14$.
\end{proof}

Note that, as a by-product, the proof of Theorem~\ref{thm:socpglobalopt} controls the extreme eigenvalues of $S$ at global optima (aside from the trivial eigenvalue corresponding to the global phase indeterminacy). It can be shown that the ratio of the largest to smallest positive eigenvalues of $S$ upper bounds
the condition number of the Riemannian Hessian of $f$ on the quotient space $\Cn_1 / \!\! \sim$. As a result, the smaller $\Delta$, the closer the condition number is to 1, and the faster the local convergence of classical Riemannian optimization algorithms~\cite[Thm.\,7.4.11]{AMS08}.

\section{Numerical experiments}\label{sec:XP}

Following the Gaussian noise setup of Lemma~\ref{lemma:Wignernoise}---identical to the experimental setup in~\cite{bandeira2014tightness}---we generate, independently for various values of $n$ and $\sigma$, 100 independent noise realizations $W$ and a uniformly random signal $z \in \Cn_1$. For each resulting data matrix $C = zz^* + \sigma W$, we compute the eigenvector estimator $z_{EIG} = \hat v$~\eqref{eq:eig}, an estimator $z_{GPM}$ computed with Algorithm~\ref{algo:ppm}, and an estimator $z_{RTR}$ computed with the Riemannian trust-region algorithm (RTR), via the Manopt toolbox~\cite{genrtr,manopt}. Figures 1--7 report aggregated statistics about the results.

The GPM estimator $z_{GPM}$ is obtained as follows. Parameter $\alpha$ is set to the smallest allowed value,
that is, $\alpha = \max(0, -\lambdamin(C))$. This is computed from data. Then, Algorithm~\ref{algo:ppm}, initialized with $x_0 = z_{EIG}$, iterates until $\frac{x_k^*\tilde C x_k}{\|\tilde C x_k\|_1} \geq 1 - 10^{-7}$. This happens in finite time since the left hand side converges to 1 from below. Average iteration counts are reported in Figure~\ref{fig:GPMiterations}.

The RTR estimator $z_{RTR}$ is obtained as follows. RTR is run out-of-the-box on~\eqref{eq:P} with cost scaled by $1/n^2$, with default parameter values and a \emph{random} initial guess until the Riemannian gradient norm drops below $10^{-6}$, that is, $2\|S(x)x\|_2/n^2 \leq 10^{-6}$. This happens in finite time since the algorithm converges to critical points~\citep{boumal2016globalrates}. Average iteration counts are reported in Figure~\ref{fig:RTRiterations}. In the form we use, RTR is not guaranteed to converge to second-order critical points, but it does so in practice as only such points are stable for the iteration.
Figure~\ref{fig:RTREIGiterations} reports average iteration counts for RTR initialized with the eigenvector estimator. The corresponding Figure~\ref{fig:RTRglobalopt} (described below) is indistinguishable, and hence omitted.

Based on Lemma~\ref{lemma:sufficientS}, global optimality at $x$ is declared (up to numerical accuracy) if $S = S(x)$ is positive semidefinite (up to numerical accuracy). This is declared to be the case if $\lambdamin(S)/|\lambdamax(S)| \geq -10^{-5}$ for $z_{GPM}$, and $-10^{-9}$ for $z_{RTR}$---the difference reflects the faster local convergence of RTR, which allows to reach higher accuracy for little extra effort. Success rates of this global optimality test appear in Figures~\ref{fig:GPMglobalopt} and~\ref{fig:RTRglobalopt} for GPM and RTR, respectively. The figures are essentially indistinguishable.

A phase transition clearly appears. The results suggest $\sigma$ up to $\tilde{\mathcal{O}}(n^{1/2})$ can be handled. Comparing with~\cite[Fig.\,2]{bandeira2014tightness}, it appears that GPM and RTR solve~\eqref{eq:P} for noise levels as large as the semidefinite relaxation can handle, even though both methods scale better than interior point methods in practice. (Figures in~\citep{bandeira2014tightness} were generated with a recent low-rank SDP solver~\citep{boumal2016bmapproach}.)

Figure~\ref{fig:RTRbeatsEIG} displays how often the RTR estimator is closer to $z$ (in the phase-aligned $\ell_2$-sense) than the eigenvector estimator. They appear to be mostly equally good estimators (with a slight advantage for RTR), except for the region where $n$ is large and $\sigma$ is close to but smaller than $\sqrt{n}$. In that challenging regime, $z_{RTR}$ consistently outperforms the simpler estimator $z_{EIG}$.

Finally, Figure~\ref{fig:EIGbeatssignal} displays how often the eigenvector estimator is a more likely estimator than the planted signal itself, that is, $\hat v^* C \hat v > z^* C z$. This figure is provided as a baseline to verify in what regime $\hat v$ is an excellent initialization for any algorithm aimed at solving~\eqref{eq:P}. In particular, initializing RTR with $\hat v$ speeds up computations in practice: compare Figures~\ref{fig:RTRiterations} and~\ref{fig:RTREIGiterations}.

\section{Perspectives and conclusions}

We showed phase synchronization as posed in~\eqref{eq:P} can be solved to global optimality with GPM, a simple algorithm operating directly in $\Cn_1$, under some conditions on the noise. This is more practical than solving~\eqref{eq:P} via semidefinite relaxation in a high dimensional space~\cite{bandeira2014tightness}. The main theorems hold under more restrictive assumptions on the noise than those made in~\cite{bandeira2014tightness}, but numerical experiments suggest GPM (and RTR) succeed in the same regime as the SDP relaxation. Similarly to~\cite{bandeira2014tightness}, the bottleneck in the analysis is Lemma~\ref{lem:Deltaxinfty}. Improving the latter would improve results in both papers.

We further showed that, under some conditions on the noise, second-order necessary optimality conditions are sufficient for~\eqref{eq:P}. In that regime, strong duality holds for~\eqref{eq:P}.
To the best of our knowledge, it is not known whether strong duality generally (for some broad class of nonconvex problems encompassing~\eqref{eq:P}) implies that second-order necessary optimality conditions become sufficient. If this is so, analyses such as presented in Section~\ref{sec:landscape} might be simplified and improved.

A natural extension of this work is to apply it for synchronization of rotations and orthogonal transformations in $\Rd$, for $d > 2$; see~\cite{carmona2012analytical,wang2012LUD,boumal2015staircase,bandeira2013approximating} among others. In this scenario, $C$ is a block-matrix, with off-diagonal blocks of size $d\times d$ being noisy measurements of relative transformations in $\Rd$.
For $d=3$, an alternative to using block matrices is to represent rotations as quaternions~\cite{carmona2012analytical}.
Experiments (not shown) suggest GPM and RTR perform well in this extended setting.

Under non-Gaussian noise, it may be useful to consider alternative cost functions. An interesting one is the \emph{least unsquared deviations} cost studied for synchronization of rotations ($d \geq 2$) by Wang and Singer~\cite{wang2012LUD}. Experiments (not shown) suggest GPM can be adapted to work effectively with this robust (but nonsmooth) cost function via the \emph{iteratively reweighted least squares} approach. This approach bears some resemblance with the Weiszfeld algorithm~\cite{hartley2011l1}. It would be interesting to study the convergence to global optimality of such methods.

Finally, it is interesting to establish rates of convergence, that is, to bound the number of iterations required to reach approximate solutions. For GPM, Figure~\ref{fig:GPMiterations} suggests this number is low in favorable noise regimes. Soon after the first appearance of the present paper (and partially in response to it), Liu et al.~\citep{liu2016statistical} obtained convergence rates for GPM in a similar setting. For RTR, worst-case iteration complexity bounds for the computation of approximate second-order critical points on manifolds are developed in~\cite{boumal2016globalrates}. The results there apply here, since $\Cn_1$ is a compact manifold and $f$ is smooth in $\Cn$. Still, even though they are sharp, the worst-case bounds seem pessimistic in view of the favorable empirical performance.

\section*{Acknowledgments}

I thank P.-A.\ Absil, A.S.\ Bandeira, Q.\ Berthet, A.\ d'Aspre\-mont, D.\ Scieur, A.\ Singer and B.\ Van\-de\-reycken for fruitful discussions, as well as the anonymous reviewers for helpful suggestions. This research was generously supported by the ``Fonds Sp\'eciaux de Recherche'' (FSR) from UCLouvain and, thanks to A.\ d'Aspre\-mont, by the Chaire Havas ``Chaire Eco\-no\-mie et gestion des nouvelles don\-n\'ees,'' the ERC Starting Grant SIPA and a Research in Paris grant.

\begin{figure}[h]
	\begin{center}
		\includegraphics[width=0.9\textwidth]{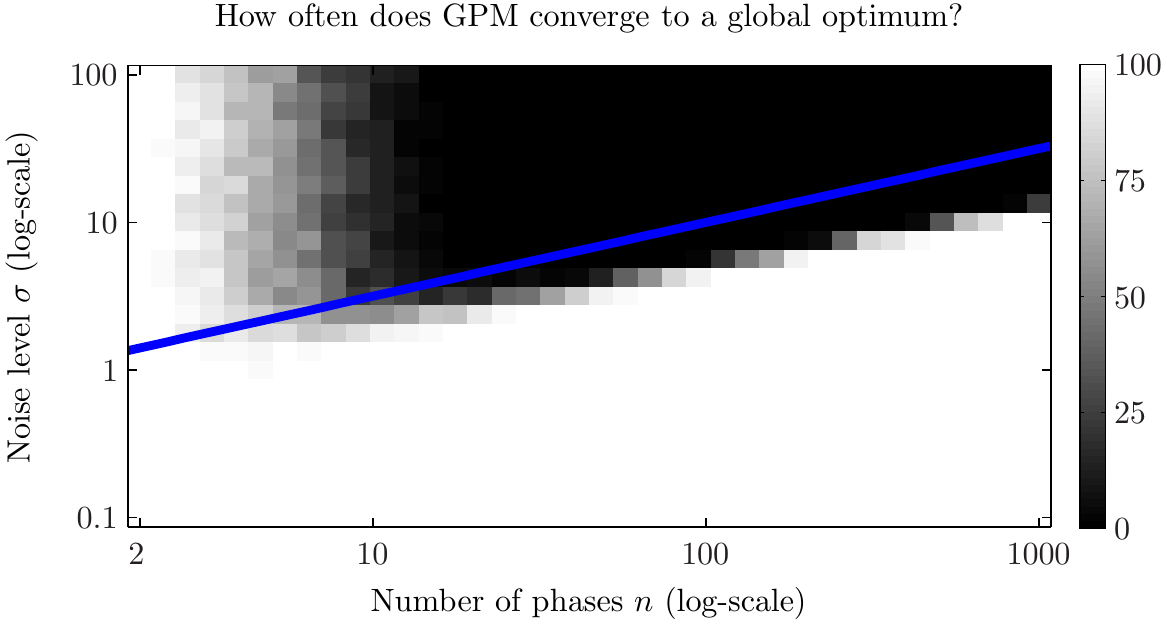}
		\caption{The generalized power method (GPM, Algorithm~\ref{algo:ppm}) identifies a global optimum of~\eqref{eq:P} even for large levels of noise. Global optimality is certified a posteriori up to some numerical tolerance via Lemma~\ref{lemma:sufficientS}. This is partly explained by Theorem~\ref{thm:ppmglobalopt}. In all figures, the blue line marks $\sigma = \sqrt{n}$.}
		\label{fig:GPMglobalopt}
	\end{center}
\end{figure}
\begin{figure}[h]
	\begin{center}
		\includegraphics[width=0.9\textwidth]{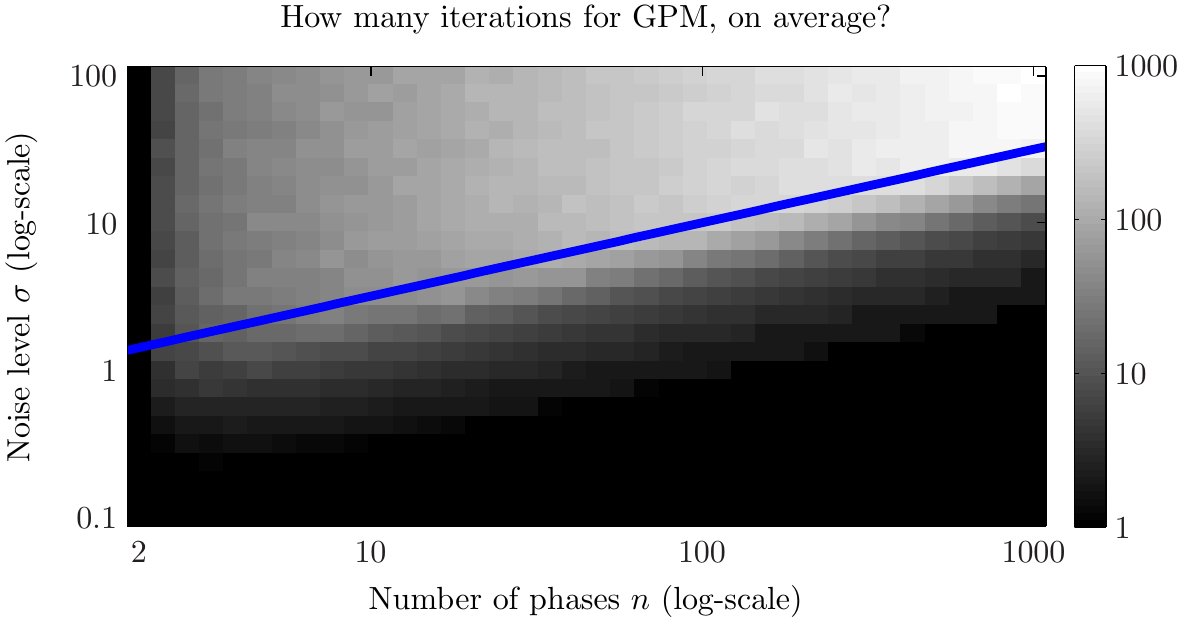}
		\caption{Iteration count of GPM until the stopping criterion triggers, averaged over the 100 repetitions for each pair $(n, \sigma)$. A typical pattern appears, where a simpler statistical task (smaller $\sigma$) translates into a better conditioned optimization problem requiring fewer iterations (see~\cite{roulet2015renegar} for a related discussion).}
		\label{fig:GPMiterations}
	\end{center}
\end{figure}

\begin{figure}[h]
	\begin{center}
		\includegraphics[width=0.9\textwidth]{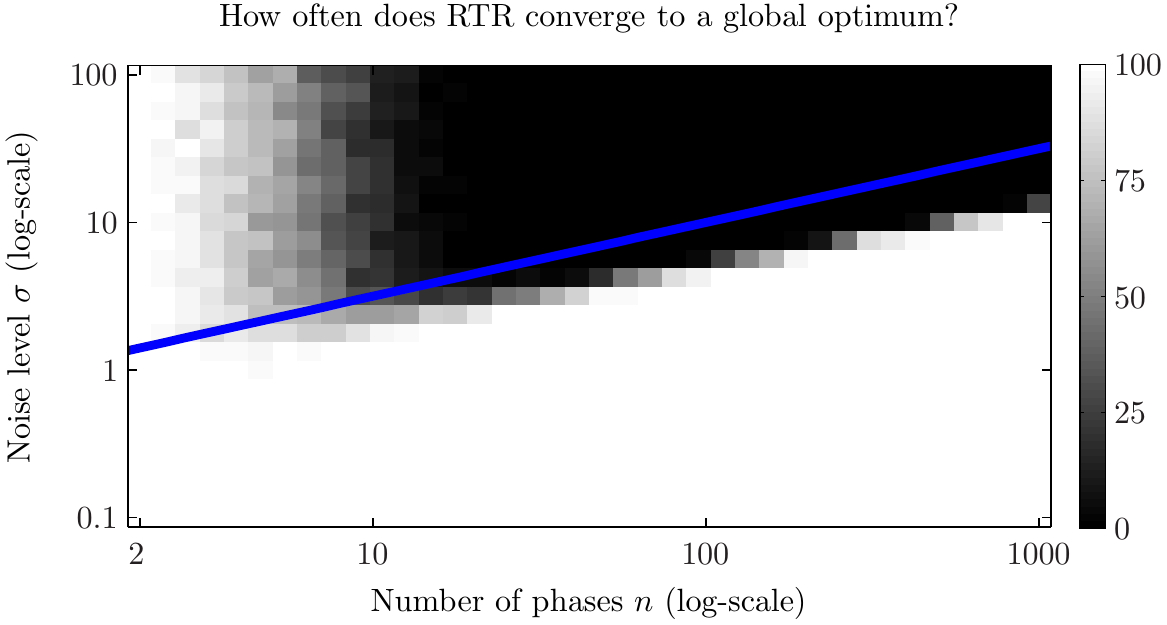}
		\caption{Same as Figure~\ref{fig:GPMglobalopt}, this time with the general purpose Riemannian trust-region algorithm (RTR) instead of GPM. Both algorithms appear to reach global optimality in the same regime. This is partly explained by Theorem~\ref{thm:socpglobalopt}.}
		\label{fig:RTRglobalopt}
	\end{center}
\end{figure}
\begin{figure}[h]
	\begin{center}
		\includegraphics[width=0.9\textwidth]{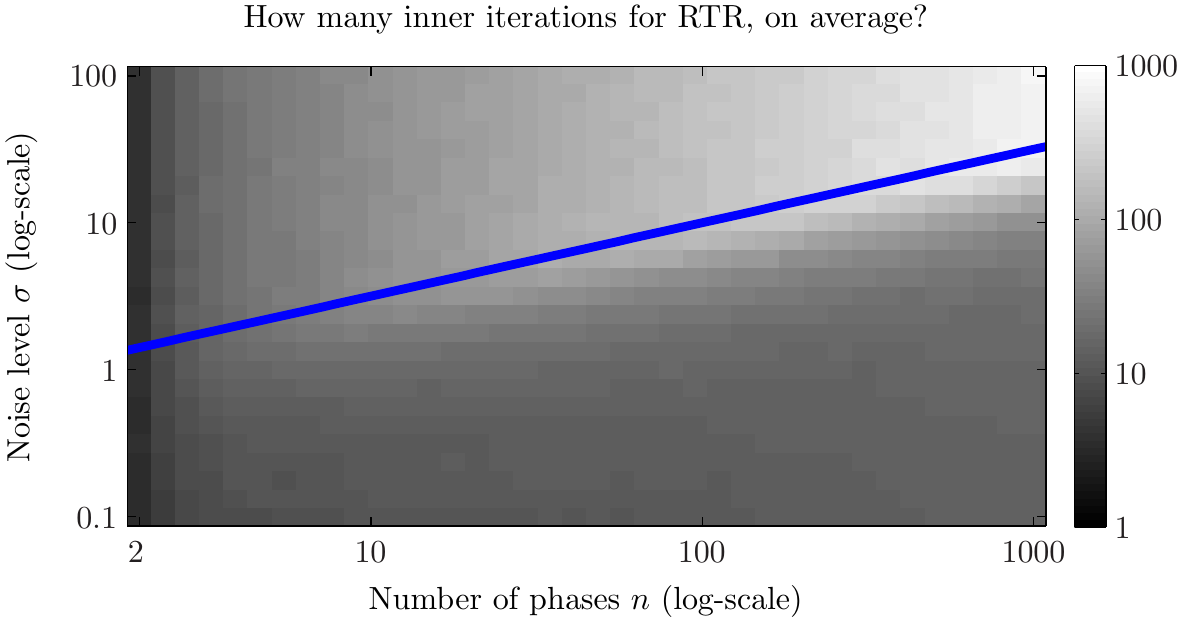}
		\caption{Same as Figure~\ref{fig:GPMiterations}, for the RTR algorithm. The sum of outer and inner iteration counts is close to the number of matrix-vector products with $C$, hence is comparable with iteration counts in GPM. RTR requires more work than GPM below the blue line, but starts from a random initial point and attains higher accuracy. RTR has the advantage of being a general purpose algorithm. On the other hand, GPM is particularly simple compared to RTR.}
		\label{fig:RTRiterations}
	\end{center}
\end{figure}
\begin{figure}[h]
	\begin{center}
		\includegraphics[width=0.9\textwidth]{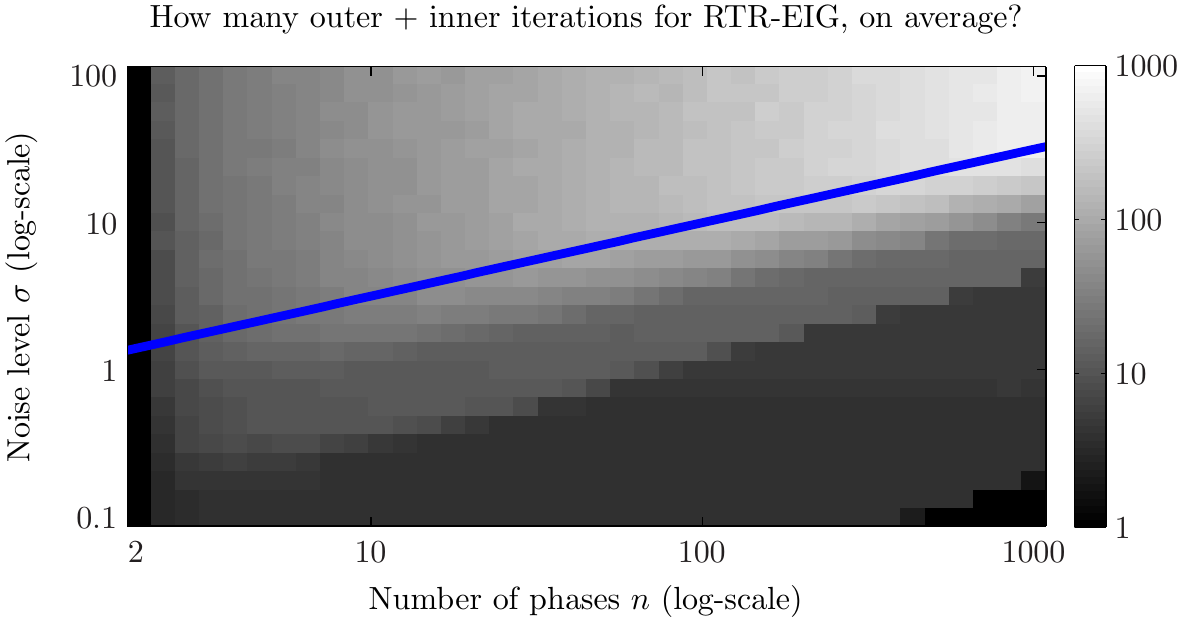}
		\caption{Same as Figure~\ref{fig:RTRiterations}, for the RTR algorithm initialized with the eigenvector estimator. Comparing with Figure~\ref{fig:RTRiterations}, the usefulness of this initialization is clear.}
		\label{fig:RTREIGiterations}
	\end{center}
\end{figure}

\begin{figure}[h]
	\begin{center}
		\includegraphics[width=0.9\textwidth]{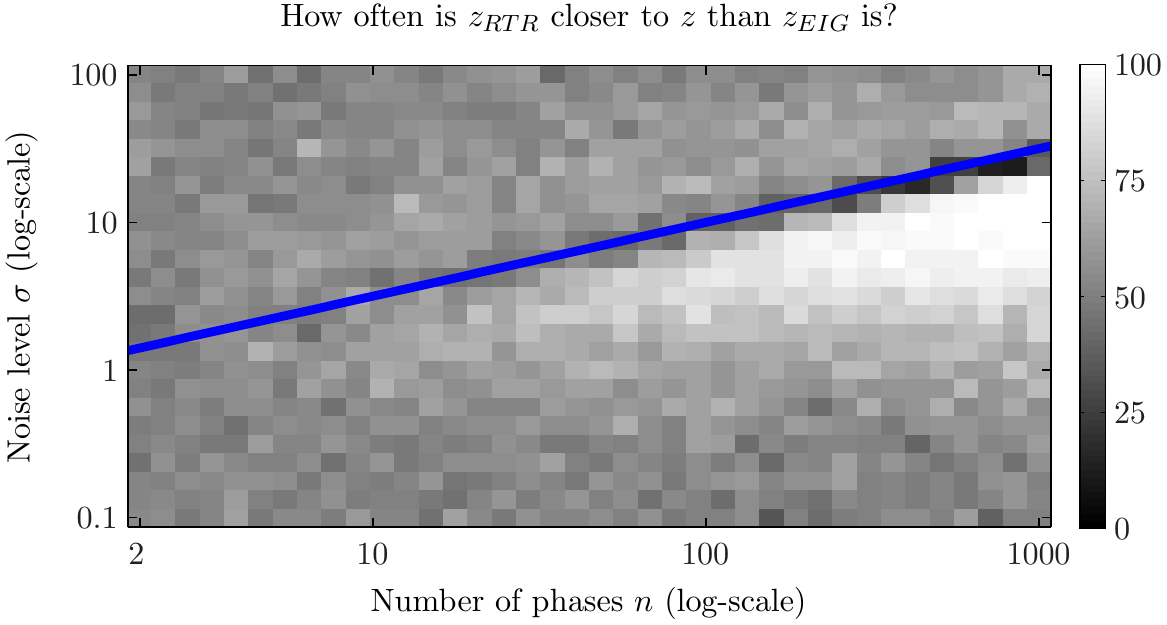}
		\caption{
			For large $n$, close to the phase transition (somewhat below the blue line $\sigma = \sqrt{n}$), $z_{RTR}$ is a better estimator than $z_{EIG}$, in that $d(z, z_{RTR}) < d(z, z_{EIG})$~\eqref{eq:distance}. For smaller noise levels, $z_{RTR}$ only has a mild advantage over $z_{EIG}$. This is consistent with Lemmas~\ref{lemma:ell2bound} and~\ref{lem:eigell2}.}
		\label{fig:RTRbeatsEIG}
	\end{center}
\end{figure}
\begin{figure}[h]
	\begin{center}
		\includegraphics[width=0.9\textwidth]{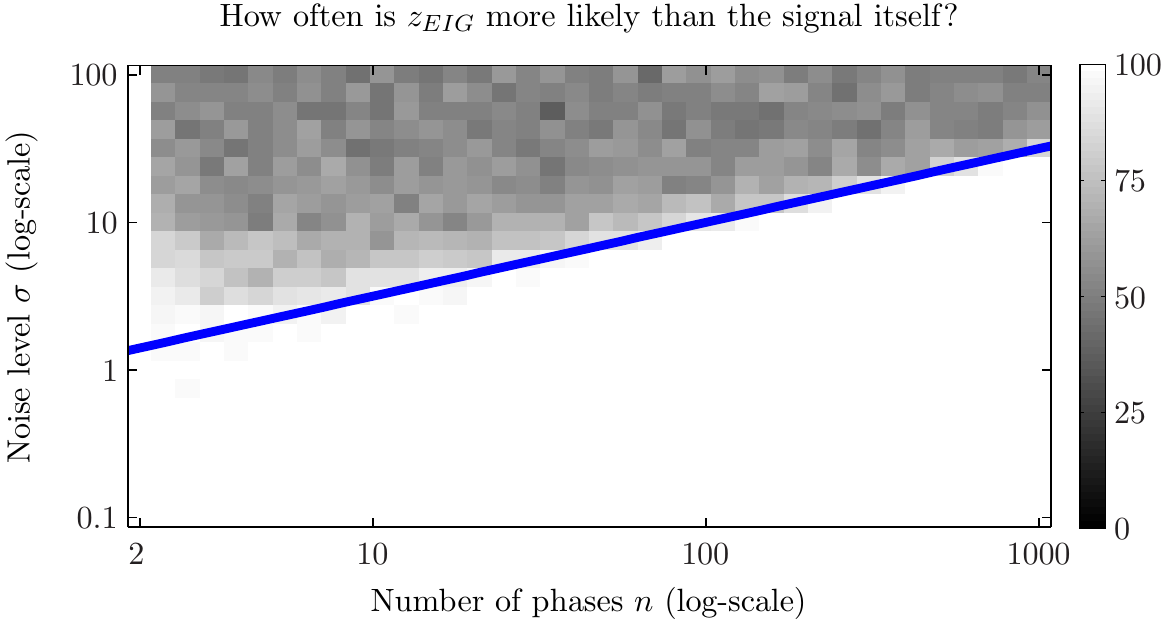}
		\caption{For noise levels $\sigma < \sqrt{n}$ (below the blue line), $z_{EIG}$ empirically attains a higher likelihood as an estimator for $z$ than $z$ itself. This confirms that, in this favorable regime, $z_{EIG}$ is an excellent initialization for any algorithm aiming to solve~\eqref{eq:P}. In particular, it is known that, with high probability, monotonic ascent from such a point to a second-order critical point results in a global optimum for $\sigma \leq n^{1/4}/18$~\cite[Prop.\,4.5]{bandeira2014tightness}.}
		\label{fig:EIGbeatssignal}
	\end{center}
\end{figure}

\clearpage
\bibliographystyle{plain}
\bibliography{../../boumal}

\begin{thebibliography}{10}

\bibitem{genrtr}
P.-A. Absil, C.~G. Baker, and K.~A. Gallivan.
\newblock Trust-region methods on {Riemannian} manifolds.
\newblock {\em Foundations of Computational Mathematics}, 7(3):303--330, 2007.

\bibitem{AMS08}
P.-A. Absil, R.~Mahony, and R.~Sepulchre.
\newblock {\em Optimization Algorithms on Matrix Manifolds}.
\newblock Princeton University Press, Princeton, NJ, 2008.

\bibitem{agrawal2006surface}
A.~Agrawal, R.~Raskar, and R.~Chellappa.
\newblock What is the range of surface reconstructions from a gradient field?
\newblock In A.~Leonardis, H.~Bischof, and A.~Pinz, editors, {\em Computer
  Vision -- ECCV 2006}, volume 3951 of {\em Lecture Notes in Computer Science},
  pages 578--591. Springer Berlin Heidelberg, 2006.

\bibitem{Alexeev_PhaseRetrievalPolarization}
B.~Alexeev, A.S. Bandeira, M.~Fickus, and D.G. Mixon.
\newblock Phase retrieval with polarization.
\newblock {\em SIAM Journal on Imaging Science}, 7(1):35--66, 2013.

\bibitem{bandeira2014tightness}
A.S. Bandeira, N.~Boumal, and A.~Singer.
\newblock Tightness of the maximum likelihood semidefinite relaxation for
  angular synchronization.
\newblock {\em Mathematical Programming}, pages 1--23, 2016.

\bibitem{bandeira2016lowrankmaxcut}
A.S. Bandeira, N.~Boumal, and V.~Voroninski.
\newblock On the low-rank approach for semidefinite programs arising in
  synchronization and community detection.
\newblock In {\em Proceedings of The 29th Conference on Learning Theory, {COLT}
  2016, New York, NY, June 23--26}, 2016.

\bibitem{Bandeira_FourierMasks}
A.S. Bandeira, Y.~Chen, and D.G. Mixon.
\newblock Phase retrieval from power spectra of masked signals.
\newblock {\em Information and Inference: a Journal of the IMA}, 3:83--102,
  2014.

\bibitem{bandeira2013approximating}
A.S. Bandeira, C.~Kennedy, and A.~Singer.
\newblock Approximating the little {G}rothendieck problem over the orthogonal
  and unitary groups.
\newblock {\em Mathematical Programming}, pages 1--43, 2016.

\bibitem{bandeira2012cheeger}
A.S. Bandeira, A.~Singer, and D.A. Spielman.
\newblock A {C}heeger inequality for the graph connection {L}aplacian.
\newblock {\em SIAM Journal on Matrix Analysis and Applications},
  34(4):1611--1630, 2013.

\bibitem{bendory2016stft}
T.~Bendory and Y.C. Eldar.
\newblock Non-convex phase retrieval from {STFT} measurements.
\newblock {\em arXiv preprint arXiv:1607.08218}, 2016.

\bibitem{boumal2015staircase}
N.~Boumal.
\newblock A {R}iemannian low-rank method for optimization over semidefinite
  matrices with block-diagonal constraints.
\newblock {\em arXiv preprint arXiv:1506.00575}, 2015.

\bibitem{boumal2016globalrates}
N.~Boumal, P.-A. Absil, and C.~Cartis.
\newblock Global rates of convergence for nonconvex optimization on manifolds.
\newblock {\em arXiv preprint arXiv:1605.08101}, 2016.

\bibitem{manopt}
N.~Boumal, B.~Mishra, P.-A. Absil, and R.~Sepulchre.
\newblock {M}anopt, a {M}atlab toolbox for optimization on manifolds.
\newblock {\em Journal of Machine Learning Research}, 15:1455--1459, 2014.

\bibitem{crbsynch}
N.~Boumal, A.~Singer, P.-A. Absil, and V.D. Blondel.
\newblock {C}ram{\'e}r-{R}ao bounds for synchronization of rotations.
\newblock {\em Information and Inference}, 3:1--39, 2014.

\bibitem{boumal2016bmapproach}
N.~Boumal, V.~Voroninski, and A.S. Bandeira.
\newblock The non-convex {B}urer-{M}onteiro approach works on smooth
  semidefinite programs.
\newblock {\em arXiv preprint arXiv:1606.04970}, 2016.

\bibitem{burer2002rank}
S.~Burer, R.D.C. Monteiro, and Y.~Zhang.
\newblock Rank-two relaxation heuristics for {Max-Cut} and other binary
  quadratic programs.
\newblock {\em SIAM Journal on Optimization}, 12(2):503--521, 2002.

\bibitem{carmona2012analytical}
M.~Carmona, O.~Michel, J.-L. Lacoume, N.~Sprynski, and B.~Nicolas.
\newblock An analytical solution for the complete sensor network attitude
  estimation problem.
\newblock {\em Signal Processing}, 93(4):652--660, 2013.

\bibitem{cucuringu2015syncrank}
M.~Cucuringu.
\newblock {Sync-Rank}: Robust ranking, constrained ranking and rank aggregation
  via eigenvector and sdp synchronization.
\newblock {\em IEEE Transactions on Network Science and Engineering},
  3(1):58--79, 2016.

\bibitem{deshpande2014conePCA}
Y.~Deshpande, A.~Montanari, and E.~Richard.
\newblock Cone-constrained principal component analysis.
\newblock In Z.~Ghahramani, M.~Welling, C.~Cortes, N.D. Lawrence, and K.Q.
  Weinberger, editors, {\em Advances in Neural Information Processing Systems
  27}, pages 2717--2725. Curran Associates, Inc., 2014.

\bibitem{ge2015escaping}
R.~Ge, F.~Huang, C.~Jin, and Y.~Yuan.
\newblock Escaping from saddle points---online stochastic gradient for tensor
  decomposition.
\newblock In {\em Proceedings of The 28th Conference on Learning Theory}, pages
  797--842, 2015.

\bibitem{giridhar2006distributed}
A.~Giridhar and P.R. Kumar.
\newblock Distributed clock synchronization over wireless networks: Algorithms
  and analysis.
\newblock In {\em Decision and Control, 2006 45th IEEE Conference on}, pages
  4915--4920. IEEE, 2006.

\bibitem{hartley2011l1}
R.~Hartley, K.~Aftab, and J.~Trumpf.
\newblock {L1} rotation averaging using the {W}eiszfeld algorithm.
\newblock In {\em {Computer Vision and Pattern Recognition (CVPR), 2011 IEEE
  Conference on}}, pages 3041--3048. IEEE, 2011.

\bibitem{hartley2013rotation}
R.~Hartley, J.~Trumpf, Y.~Dai, and H.~Li.
\newblock Rotation averaging.
\newblock {\em International Journal of Computer Vision}, 103(3):267--305,
  2013.

\bibitem{howard2010estimation}
S.D. Howard, D.~Cochran, W.~Moran, and F.R. Cohen.
\newblock Estimation and registration on graphs.
\newblock {\em Arxiv preprint arXiv:1010.2983}, 2010.

\bibitem{javanmard2015phase}
A.~Javanmard, A.~Montanari, and F.~Ricci-Tersenghi.
\newblock Phase transitions in semidefinite relaxations.
\newblock {\em arXiv preprint arXiv:1511.08769}, 2015.

\bibitem{journee2010generalized}
M.~Journ{\'e}e, Y.~Nesterov, P.~Richt{\'a}rik, and R.~Sepulchre.
\newblock Generalized power method for sparse principal component analysis.
\newblock {\em The Journal of Machine Learning Research}, 11:517--553, 2010.

\bibitem{liu2016statistical}
H.~Liu, M.-C. Yue, and A.~M.-C. So.
\newblock On the statistical performance of the generalized power method for
  angular synchronization.
\newblock {\em arXiv preprint arXiv:1603.00211}, 2016.

\bibitem{liu2015snig}
X.~Liu, H.~Wang, X.~Chen, and Y.~Yuan.
\newblock On the global optimality for linear constrained rank minimization
  problem.
\newblock Optimization Online, 2015.

\bibitem{luss2013conditional}
R.~Luss and M.~Teboulle.
\newblock Conditional gradient algorithms for rank-one matrix approximations
  with a sparsity constraint.
\newblock {\em SIAM Review}, 55(1):65--98, 2013.

\bibitem{matinec2007rotation}
D.~Martinec and T.~Pajdla.
\newblock Robust rotation and translation estimation in multiview
  reconstruction.
\newblock In {\em Computer Vision and Pattern Recognition, 2007. CVPR '07. IEEE
  Conference on}, pages 1--8, June 2007.

\bibitem{montanari2014nonnegativepca}
A.~Montanari and E.~Richard.
\newblock Non-negative principal component analysis: Message passing algorithms
  and sharp asymptotics.
\newblock {\em arXiv preprint arXiv:1406.4775}, 2014.

\bibitem{netrapalli2015altminphase}
P.~Netrapalli, P.~Jain, and S.~Sanghavi.
\newblock Phase retrieval using alternating minimization.
\newblock {\em Signal Processing, IEEE Transactions on}, 63(18):4814--4826,
  2015.

\bibitem{roulet2015renegar}
V.~Roulet, N.~Boumal, and A.~d'Aspremont.
\newblock {R}enegar's condition number and compressed sensing performance.
\newblock {\em arXiv preprint arXiv:1506.03295}, 2015.

\bibitem{rubinstein2001optical}
J.~Rubinstein and G.~Wolansky.
\newblock Reconstruction of optical surfaces from ray data.
\newblock {\em Optical Review}, 8(4):281--283, 2001.

\bibitem{singer2010angular}
A.~Singer.
\newblock Angular synchronization by eigenvectors and semidefinite programming.
\newblock {\em Applied and Computational Harmonic Analysis}, 30(1):20--36,
  2011.

\bibitem{singer2011eigen}
A.~Singer and Y.~Shkolnisky.
\newblock {Three-Dimensional Structure Determination from Common Lines in
  Cryo-EM by Eigenvectors and Semidefinite Programming}.
\newblock {\em SIAM Journal on Imaging Sciences}, 4(2):543--572, 2011.

\bibitem{So_2010_SDP_detector}
A.M.C. So.
\newblock Probabilistic analysis of the semidefinite relaxation detector in
  digital communications.
\newblock {\em Proceedings of SODA}, 2010.

\bibitem{soltanalian2014designing}
M.~Soltanalian and P.~Stoica.
\newblock Designing unimodular codes via quadratic optimization.
\newblock {\em Signal Processing, IEEE Transactions on}, 62(5):1221--1234,
  2014.

\bibitem{sun2015complete}
J.~Sun, Q.~Qu, and J.~Wright.
\newblock Complete dictionary recovery over the sphere.
\newblock {\em arXiv preprint arXiv:1504.06785}, 2015.

\bibitem{sun2015nonconvex}
J.~Sun, Q.~Qu, and J.~Wright.
\newblock When are nonconvex problems not scary?
\newblock {\em arXiv preprint arXiv:1510.06096}, 2015.

\bibitem{sun2016geometric}
J.~Sun, Q.~Qu, and J.~Wright.
\newblock A geometric analysis of phase retrieval.
\newblock {\em arXiv preprint arXiv:1602.06664}, 2016.

\bibitem{vavasis1991nonlinear}
S.A. Vavasis.
\newblock {\em Nonlinear optimization: complexity issues}.
\newblock Oxford University Press, Inc., 1991.

\bibitem{waldspurger2012phase}
I.~Waldspurger, A.~d'Aspremont, and S.~Mallat.
\newblock Phase recovery, {MaxCut} and complex semidefinite programming.
\newblock {\em Mathematical Programming}, 149(1--2):47--81, 2015.

\bibitem{wang2012LUD}
L.~Wang and A.~Singer.
\newblock Exact and stable recovery of rotations for robust synchronization.
\newblock {\em Information and Inference}, 2(2):145--193, 2013.

\bibitem{zhang2006complex}
S.~Zhang and Y.~Huang.
\newblock Complex quadratic optimization and semidefinite programming.
\newblock {\em SIAM Journal on Optimization}, 16(3):871--890, 2006.

\bibitem{zhao2015nonconvex}
T.~Zhao, Z.~Wang, and H.~Liu.
\newblock A nonconvex optimization framework for low rank matrix estimation.
\newblock In {\em Advances in Neural Information Processing Systems}, pages
  559--567, 2015.

\end{thebibliography}

\appendix

\section{Connectedness of the set of accumulation points} \label{apdx:connectedness}

Similarly to~\cite[Thm.\,4]{journee2010generalized}, convergence of Algorithm~\ref{algo:ppm} can be further controlled after Lemma~\ref{lem:ppmaccumpts} if we accept to make $\tilde C$ \emph{strictly} positive definite by increasing $\alpha$. This inertia increase has the counter-effect of slowing down convergence and strengthening the conditions on $\Delta$ in the lemmas and theorems of Section~\ref{sec:ppm}. These counter-effects are mitigated by the fact that the increase in $\alpha$ can be taken arbitrarily small.
\begin{lemma}\label{lem:connectedness}
	If $\lambda \triangleq \lambda_{\min}(\tilde C) > 0$, then for all $K \geq 1$,
	\begin{align}
	\sum_{k = 0}^{K-1} \|x_{k+1} - x_k\|_2^2 \leq \frac{f(x_K)-f(x_0)}{\lambda} \leq \frac{f_\infty-f(x_0)}{\lambda},
	\label{eq:connectedness}
	\end{align}
	where $f_\infty = \lim_{k\to\infty} f(x_k)$. As a result, the set of accumulation points of $x_0, x_1, \ldots$ is connected, and all of them are fixed points.
\end{lemma}
\begin{proof}
	By strong convexity of $g(x) = x^*\tilde C x$, it holds for all $y\in\Cn$ that
	\begin{align*}
	g(y) - g(x_k) \geq 2\innersmall{y - x_k}{\tilde Cx_k} + \lambda \|y-x_k\|_2^2.
	\end{align*}
	In particular, for $y = T(x_k) = x_{k+1}$ which maximizes $y\mapsto\innersmall{y - x_k}{\tilde Cx_k}$ over $\Cn_1$ and makes it nonnegative, it further holds that
	\begin{align*}
	g(x_{k+1}) - g(x_k) \geq \lambda \|x_{k+1}-x_k\|_2^2.
	\end{align*}
	Sum for $k = 0, \ldots, K-1$ to establish~\eqref{eq:connectedness}.
	
	For contradiction,\footnote{The argument is based on a post by Leo on Mathematics Stack Exchange, question 848884.}
	assume the set $\omega(x_0)$ of accumulation points of $(x_k)_{k=0,1\ldots}$ (its $\omega$-limit set) is disconnected. Then, since $\omega(x_0)$ is closed (it can be defined as the intersection of a countably infinite number of closed sets), there exist two disjoint open sets $A, B \subset \Cn_1$ such that $\omega(x_0) \subset A \cup B$, there exist $a \in A \cap \omega(x_0)$ and $b\in B \cap \omega(x_0)$, and $d \triangleq \inf_{a \in A, b \in B} \|a-b\|_2 > 0$. From~\eqref{eq:connectedness}, there exists $K$ such that $k > K \implies \|x_{k+1}-x_{k}\|_2 < d$, that is, past this index, there can be no ``jump'' from $A$ to $B$ in a single step, and vice versa. Let $K < a_1 < a_2 < \cdots$ index a converging subsequence such that $x_{a_1}, x_{a_2}\ldots$ converges to $a$, with $x_{a_k} \in A$ for all $k$, and similarly for $K < b_1 < b_2 < \cdots$. Discard elements of these subsequences to ensure $a_1 < b_1 < a_2 < b_2 < \cdots$: this does not affect their limit. Since $\|x_{a_k} - x_{b_k}\|_2 \geq d$, there must exist $a_k < c_k < b_k$ such that the sequence $x_{c_1}, x_{c_2}\ldots$ lives in $\Cn_1 \backslash (A \cup B)$, which is compact since $A,B$ are open. Hence, $x_{c_1}, x_{c_2}\ldots$ admits a subsequence converging in $\Cn_1 \backslash (A \cup B)$. This is a contradiction because, by definition, all subsequences of $x_1, x_2\ldots$ converge in $\omega(x_0) \subset A \cup B$.

	That all accumulation points are fixed points of $T$ follows from Lemma~\ref{lem:ppmaccumpts}.
\end{proof}

\section{Technical convergence lemma}

\begin{lemma} \label{lem:cvgcelemma}
	Let $\calM$ be a nonempty, compact metric space and let $f \colon \calM \to \reals$ be a continuous function on $\calM$. Let $x^*\in\calM$ be such that $f(x) = f(x^*) \iff x = x^*$. If $(x_k)_{k=0,1,\ldots}$ is a sequence such that $\lim_{k\to\infty} f(x_k) = f(x^*)$, then $(x_k)_{k=0,1,\ldots}$ converges to $x^*$.
\end{lemma}
\begin{proof}
	Let $t_0 < t_1 < \cdots$ index any convergent subsequence of $(x_k)_{k=0,1,\ldots}$ (since $\calM$ is compact, one must exist). By continuity of $f$,
	\begin{align*}
		f(x^*) & = \lim_{k\to\infty} f(x_k) = \lim_{k\to\infty} f(x_{t_k}) = f\left(\lim_{k\to\infty} x_{t_k} \right).
	\end{align*}
	By assumption on $f$, this implies $\lim_{k\to\infty} x_{t_k} = x^*$. Hence, $x^*$ is the unique accumulation point of $(x_k)_{k=0,1,\ldots}$. Thus, if $(x_k)_{k=0,1,\ldots}$ converges, it does so to $x^*$. Convergence is equivalent to the following statement ($d$ is a distance on $\calM$):
	\begin{align*}
		\forall \epsilon > 0, \exists K \textrm{ such that } \forall k > K, d(x^*, x_k) < \epsilon.
	\end{align*}
	For contradiction, assume $(x_k)_{k=0,1,\ldots}$ does not converge. Then,
	\begin{align*}
		\exists \epsilon > 0 \textrm{ such that } \forall K, \exists k > K \textrm{ such that } d(x^*, x_k) \geq \epsilon.
	\end{align*}
	Thus, we may extract a subsequence $x_{s_0}, x_{s_1}, \ldots$ such that $d(x^*, x_{s_k}) \geq \epsilon$ for all $k$. Since $\{ x \in \calM : d(x^*, x) \geq \epsilon \}$ is compact, the latter subsequence admits a subsequence converging in the latter set. But this is impossible, since all convergent subsequences of $(x_k)_{k=0,1,\ldots}$ converge to $x^*$. Hence, $(x_k)_{k=0,1,\ldots}$ converges to $x^*$.
\end{proof}

\section{Projection to $\Cn_1$}\label{apdx:eigimprovement}

We give a superficially different statement and proof of~\citep[Lemma~2]{liu2016statistical}, showing that $\ell_2$ distance to $z\in\Cn_1$ increases by at most a factor of 2 after projection from $\Cn$ to $\Cn_1$.
\begin{lemma}
	Let $z \in \Cn_1$ and $v \in \Cn$. Let $\hat v \in \Cn_1$ be a projection of $v$ to $\Cn_1$, as
	\begin{align*}
		\hat v_i & = \begin{cases}
		\frac{v_i}{|v_i|} & \textrm{ if } v_i \neq 0, \\
		a_i & \textrm{ otherwise},
		\end{cases}
	\end{align*}
	where $a \in \Cn_1$ is arbitrary. Then,
	\begin{align*}
		\|\hat v - z\|_2 \leq 2 \|v - z\|_2.
	\end{align*}
\end{lemma}
\begin{proof}
	We show the bound holds for individual entries, that is,
	\begin{align*}
		\forall i, \quad |\hat v_i - z_i| \leq 2|v_i - z_i|.
	\end{align*}
	This is certainly true for $i$ such that $v_i = 0$. For $i$ such that $v_i \neq 0$, let $v_i = re^{i\theta} z_i$ with $r \geq 0$. Then,
	\begin{align*}
		|\hat v_i - z_i|^2 & = |e^{i\theta} - 1|^2 = (\cos \theta - 1)^2 + (\sin \theta)^2 = 2(1-\cos\theta), \textrm{ and} \\
		|v_i - z_i|^2 & = |re^{i\theta} - 1|^2 = (r\cos\theta-1)^2 + (r\sin\theta)^2 = 1 + r^2 - 2r\cos\theta.
	\end{align*}
	Minimizing the last quantity with respect to $r \geq 0$ yields $r = \cos\theta$ if $\cos\theta \geq 0$ and $r = 0$ otherwise, so that
	\begin{align*}
		|v_i - z_i| \geq f_2(\theta) \triangleq \min_{r\geq 0} \sqrt{1 + r^2 - 2r\cos\theta} = \begin{cases}
		\sqrt{1-(\cos\theta)^2} = |\sin\theta| & \textrm{ if } \cos\theta \geq 0, \\
		1 & \textrm{ otherwise.}
		\end{cases}
	\end{align*}
	Defining $|\hat v_i - z_i| = f_1(\theta) \triangleq \sqrt{2(1-\cos\theta)}$, it is easy to verify that $$|\hat v_i - z_i| = f_1(\theta) \leq 2f_2(\theta) \leq 2|v_i - z_i|,$$ which concludes the proof.
\end{proof}

\end{document}